\documentclass[10pt]{article}
\usepackage{amsmath}
\usepackage{amsthm}
\usepackage{amssymb}
\newcommand{\R}{\ensuremath{\mathbb{R}}}

\begin{document}
\title{Borderline variational problems involving fractional Laplacians and critical singularities}
\author{Nassif Ghoussoub\footnote{This research was partially supported by a grant from the Natural Science and Engineering Research Council of Canada (NSERC). It was done while the authors were visiting the Fields Institute for Research in  Mathematical Sciences (Toronto). } \quad and \quad Shaya Shakerian\footnote{This work is part of a PhD thesis prepared by this author under the supervision of N. Ghoussoub.} \\
{\it\small Department of Mathematics}\\
{\it\small  University of British Columbia}\\
{\it\small Vancouver BC Canada V6T 1Z2}\\
\\\\
\date{Revised May 10, 2015}}

 \maketitle

\theoremstyle{definition}
\newtheorem{definition}{definition}[section]
\newtheorem{theorem}[definition]{Theorem}
\newtheorem{lemma}[definition]{Lemma}
\newtheorem{proposition}[definition]{Proposition}
\newtheorem{corollary}[definition]{Corollary}
\newtheorem{remark}[definition]{Remark}
\theoremstyle{definition}
\newtheorem{example}[definition]{Example}
\newtheorem{step}{Step}
\newtheorem{claim}[definition]{Claim}
\newtheorem{case}{Case}

\begin{abstract} We consider the problem of attainability of the best constant in the following critical fractional Hardy-Sobolev inequality:
\begin{equation*} 
\mu_{\gamma,s}(\R^n):= \inf\limits_{u \in H^{\frac{\alpha}{2}} (\R^n)\setminus \{0\}} \frac{ \int_{\R^n} |({-}{ \Delta})^{\frac{\alpha}{4}}u|^2 dx - \gamma \int_{\R^n} \frac{|u|^2}{|x|^{\alpha}}dx }{(\int_{\R^n} \frac{|u|^{2_{\alpha}^*(s)}}{|x|^{s}}dx)^\frac{2}{2_{\alpha}^*(s)}},
\end{equation*}
where  $0\leq s<\alpha<2$, $n>\alpha$, 
${2_{\alpha}^*(s)}:=\frac{2(n-s)}{n-{\alpha}},$ and  $\gamma \in \mathbb{R}$. This allows us to establish the existence of nontrivial weak solutions for the following doubly critical problem 
on $\R^n$,  
 \begin{equation*} 
\left\{\begin{array}{lll}
({-}{ \Delta})^{\frac{\alpha}{2}}u- \gamma \frac{u}{|x|^{\alpha}}&= |u|^{2_{\alpha}^*-2} u + {\frac{|u|^{2_{\alpha}^*(s)-2}u}{|x|^s}} & \text{in }  {\R^n}\\
\hfill u&>0  & \text{in }  \R^n,
\end{array}\right.
\end{equation*}
where  
$2_{\alpha}^*:=\frac{2 n}{n-{\alpha}}$ is the critical $\alpha$-fractional Sobolev exponent, and $\gamma < \gamma_H:=2^\alpha \frac{\Gamma^2(\frac{n+\alpha}{4})}{\Gamma^2(\frac{n-\alpha}{4})}$, the latter being the best fractional Hardy constant on $\R^n$.

\end{abstract}

\section{Introduction}
\numberwithin{equation}{section}
We consider the problem of existence of nontrivial  weak solutions to the following doubly critical problem 
on $\R^n$ involving the Fractional Laplacian:

 \begin{equation}\label{Main problem}
\left\{\begin{array}{lll}
({-}{ \Delta})^{\frac{\alpha}{2}}u- \gamma \frac{u}{|x|^{\alpha}}&= |u|^{2_{\alpha}^*-2} u + {\frac{|u|^{2_{\alpha}^*(s)-2}u}{|x|^s}} & \text{in }  {\R^n}\\
\hfill u&>0  & \text{in }  \R^n,
\end{array}\right.
\end{equation}
where  $0\leq s<\alpha<2$, $n>\alpha$, 
$2_{\alpha}^*:=\frac{2 n}{n-{\alpha}},$ ${2_{\alpha}^*(s)}:=\frac{2(n-s)}{n-{\alpha}},$  $\gamma \in \mathbb{R}$. The fractional Laplacian $({-}{ \Delta})^{\frac{\alpha}{2}}$ is  defined
 on the Schwartz class (space of rapidly decaying $C^\infty$ functions in $\R^n$) through the Fourier transform,
$$
 (-\Delta)^{\frac{\alpha}{2}}u= \mathcal{F}^{-1}(|\xi|^{\alpha}(\mathcal{F}u)) \quad \forall\xi\in\R^n,
 $$
where $ \mathcal{F}u$ denotes the Fourier transform of $u$, $\mathcal{F}u(\xi)=\int_{\R^n} e^{-2\pi i x.\xi} u(x) dx$. See \cite{Hitchhikers guide} and references therein for the basics on the fractional Laplacian. 

Problems involving two non-linearities have been studied in the case of local operators such as the Laplacian $-\Delta$, the $p$-Laplacian $-\Delta_p$ and the Biharmonic operator $\Delta^2$ (See \cite{Bhakta},  \cite{Filippucci-Pucci-Robert}, \cite{Kang-Li} and \cite{Xuan-Wang}). Problem (\ref{Main problem}) above is the non-local counterpart of the one studied by Filippucci-Pucci-Robert in \cite{Filippucci-Pucci-Robert}, who treated the case of the $p$-Laplacian in an equation involving both the Sobolev and the Hardy-Sobolev critical exponents.

Questions of existence and non-existence of  solutions for fractional elliptic equations with singular potentials were recently studied by several authors. All studies focus, however, on problems with only one critical exponent --mostly the non-linearity $u^{2_{\alpha}^*-1} $-- and to a lesser extent the critical Hardy-Sobolev singular term ${\frac{u^{2_{\alpha}^*(s)-1}}{|x|^s}}$ (see \cite{Cotsiolis-Tavoularis}, \cite{Fall-Minlend-Thiam}, \cite{Yang} and the references therein). These cases were also studied on smooth bounded domains (see for example  \cite{B-C-D-S 2}, \cite{B-C-D-S 1}, \cite{Barrios-Medina-Peral}, \cite{Fall}, \cite{Servadei} and the references therein). In general, the case of two critical exponents involve more subtleties and difficulties, even for local differential operators.

The variational approach that we adopt here, relies on the  following fractional Hardy-Sobolev type inequality:
\begin{equation} 
C(\int_{\mathbb{R}^n} \frac{|u|^{2_{\alpha}^*(s)}}{|x|^{s}}dx)^\frac{2}{{2_{\alpha}^*(s)}} \leq \int_{\R^n} |({-}{ \Delta})^{\frac{\alpha}{4}}u|^2 dx  - \gamma \int_{\mathbb{R}^n} \frac{|u|^{2}}{|x|^{\alpha}}dx \quad \hbox{for all $u \in H^{\frac{\alpha}{2}}(\R^n)$},
\end{equation}
where $\gamma < \gamma_H:=2^\alpha \frac{\Gamma^2(\frac{n+\alpha}{4})}{\Gamma^2(\frac{n-\alpha}{4})}$  is the best fractional Hardy constant on $\R^n$. The fractional space $H^{\frac{\alpha}{2}}(\R^n)$ is defined as the completion of $C_0^{\infty}(\R^n)$ under the norm 
$$\|u\|_{H^{\frac{\alpha}{2}}(\mathbb{R}^n)}^2= \int_{\mathbb{R}^n}|2\pi \xi |^{\alpha} |\mathcal{F}u(\xi)|^2 d\xi =\int_{\mathbb{R}^n} |(-\Delta)^{\frac{\alpha}{4}}u|^2 dx.$$

The best constant in the above fractional Hardy-Sobolev inequality is defined as:

\begin{equation} \label{Problem: the best fractional Hardy-Sobolev type constant }
\mu_{\gamma,s}(\R^n):= \inf\limits_{u \in H^{\frac{\alpha}{2}} (\R^n)\setminus \{0\}} \frac{ \int_{\R^n} |({-}{ \Delta})^{\frac{\alpha}{4}}u|^2 dx - \gamma \int_{\R^n} \frac{|u|^2}{|x|^{\alpha}}dx }{(\int_{\R^n} \frac{|u|^{2_{\alpha}^*(s)}}{|x|^{s}}dx)^\frac{2}{2_{\alpha}^*(s)}}.
\end{equation}
One step towards addressing Problem (\ref{Main problem}) consists of proving the existence of extremals for $\mu_{\gamma,s}(\R^n),$ when $s \in [0,\alpha)$ and $ \gamma\in (-\infty, \gamma_H).$ Note that the Euler-Lagrange equation corresponding to the minimization problem for $\mu_{\gamma,s}(\R^n)$ is --up to a constant factor-- the following:
 \begin{equation}\label{one}
\left\{\begin{array}{rl}
({-}{ \Delta})^{\frac{\alpha}{2}}u- \gamma \frac{u}{|x|^{\alpha}}= {\frac{u^{2_{\alpha}^*(s)-1}}{|x|^s}} & \text{in }  {\R^n}\\
 u>0 \,\,\,\,\,\,\,\,\,\,\,\,\,\,\,\,\,  & \text{in }  \mathbb{R}^n.
\end{array}\right.
\end{equation}
When $\alpha=2$, i.e., in the case of the standard Laplacian, the above minimization problem (\ref{Problem: the best fractional Hardy-Sobolev type constant }) has been extensively studied. See for example  \cite{Catrina-Wang}, \cite{Chern-Lin}, \cite{Filippucci-Pucci-Robert}, \cite{Ghoussoub-Moradifam}, \cite{Ghoussoub-Robert 2014} and \cite{Ghoussoub-Yuan}.

The non-local case has also been the subject of several studies, but in the absence of the Hardy term, i.e., when $\gamma=0$. 
In \cite{Fall-Minlend-Thiam}, Fall, Minlend and Thiam proved the existence of extremals for $\mu_{0,s}(\R^n)$ in the case $\alpha=1.$ Recently, J. Yang in \cite{Yang} proved that there exists a positive, radially symmetric and non-increasing extremal for $\mu_{0,s}(\R^n)$ when $\alpha \in (0,2).$ Asymptotic properties of the positive solutions were  given by  Y. Lei  \cite{Lei}, Lu and Zhu \cite{Lu-Zhu}, and  Yang and  Yu \cite{Yang-Yu}. 

In section 3, we consider the remaining cases in the problem of deciding  whether the best constant in the fractional Hardy-Sobolev inequality  $\mu_{\gamma,s}(\R^n)$ is attained.
 We use Ekeland's variational principle to show the following.

\begin{theorem} \label{Theorem the best fractional H-S constan}  
Suppose $0<\alpha<2$, $ 0 \le s < \alpha<n$, and $\gamma < \gamma_H := 
2^\alpha \frac{\Gamma^2(\frac{n+\alpha}{4})}{\Gamma^2(\frac{n-\alpha}{4})}$. 
\begin{enumerate}
\item If either $ \{ s > 0 \} \text{ or } \{ s=0 \text{ and } \gamma \ge 0 \}$, then $\mu_{\gamma,s}(\R^n)$ is attained.

\item If $s=0$ and $\gamma < 0$, then there are no extremals for $\mu_{\gamma,s}(\R^n).$

\item If either $\{0 < \gamma < \gamma_H  \} \text{ or } \{ 0<s<\alpha \text{ and } \gamma=0 \},$
 then any non-negative minimizer for $\mu_{\gamma,s}(\R^n)$ is positive, radially symmetric, radially decreasing,  
 and approaches zero as ${|x| \to \infty}.$
 \end{enumerate} 
\end{theorem}
 In section 4, we consider problem  (\ref{Main problem}) and use the mountain pass lemma to establish the following result.
 
  \begin{theorem}\label{Theorem Main result}
Let $0<\alpha<2,$ 
$ 0 < s < \alpha<n$ and $ 0\le \gamma < \gamma_H.$  Then, there exists a nontrivial  weak solution of (\ref{Main problem}).
\end{theorem}
Recall that $u \in H^{\frac{\alpha}{2}}(\R^n)$ is a weak solution of (\ref{Main problem}), if we have for all $\varphi \in H^{\frac{\alpha}{2}}(\R^n),$
$$\int_{\R^n}({-}{ \Delta})^{\frac{\alpha}{4}}u ({-}{ \Delta})^{\frac{\alpha}{4}} \varphi  dx = \int_{\R^n} \gamma \frac{u}{|x|^{\alpha}} \varphi dx + \int_{\R^n}|u|^{2_{\alpha}^*-2} u \varphi dx + \int_{\R^n} {\frac{|u|^{2_{\alpha}^*(s)-2}}{|x|^s}} u \varphi dx.
$$
The standard strategy to construct weak solutions of (\ref{Main problem}) is to find critical points of the corresponding functional on $H^{\frac{\alpha}{2}}(\R^n)$. 
However, (\ref{Main problem}) is invariant under the following conformal one parameter transformation group, 
\begin{equation} \label{the conformal invariance property of Main problem}
\hbox{$ T_r: H^{\frac{\alpha}{2}}(\R^n) \rightarrow H^{\frac{\alpha}{2}}(\R^n); \qquad  u(x) \rightarrow T_r[u](x)= r^{\frac{n-\alpha}{2}} u(rx), \quad r>0,$}
\end{equation}
which means that the convergence of Palais-Smale sequences is not a given. 
As it was argued in \cite{Filippucci-Pucci-Robert}, there is an asymptotic competition between the energy carried by the two critical nonlinearities. Hence, the crucial step here is to balance the competition 
 to avoid the domination of one term over another. Otherwise, there is vanishing of the weakest one, leading to a solution for the same equation but with only one critical nonlinearity. 
 In order to deal with this issue, we choose a suitable minimax energy level, in such a way that 
 after a careful analysis of the concentration phenomena,  we could eliminate the possibility of a vanishing weak limit for these well chosen Palais-Smale sequences, while ensuring 
that none of the two nonlinearities  dominate the other.

\section{Preliminaries and a description of the functional setting}

We start by recalling and  introducing suitable function spaces for the variational principles that will be needed in the sequel. We first recall the following useful representation given in \cite{B-C-D-S 2} and \cite{B-C-D-S 1}
for the fractional Laplacian $(-\Delta)^{\frac{\alpha}{2}}$ as a trace class operator, as well as for the space  $H^{\frac{\alpha}{2}}(\R^n)$.  

For a function $u \in H^{\frac{\alpha}{2}}(\R^n)$, let  $w=E_{\alpha}(u)$ be its $\alpha$-harmonic extension to the upper half-space, $\R_+^{n+1}$, that is the solution to the following problem:
 \begin{equation*}
\left\{\begin{array}{rl}
{\rm div}\,(y^{1-{\alpha}} \nabla w)=0 & \text{in } \mathbb{R}_+^{n+1} \\
w= u & \text{on }   \mathbb{R}^n \times \{y=0\}. 
\end{array}\right .
\end{equation*}
Define the space  $X^{\alpha} (\R_+^{n+1})$ as the closure of  $
C_0^{\infty}(\overline{\mathbb{R}_+^{n+1})}$ for the norm 
$$\| w\|_{X^{\alpha}({\mathbb{R}_+^{n+1}})}:=\left( k_{\alpha} \int_{\mathbb{R}_+^{n+1}} y^{1-\alpha} | \nabla w |^2 dxdy \right)^\frac{1}{2},$$
 where $k_\alpha=\frac{\Gamma(\frac{\alpha}{2})}{2^{1-\alpha}\Gamma(1-{\frac{\alpha}{2}})}$ is a normalization constant chosen in such a way that the extension operator $E_\alpha (u): {H^{\frac{\alpha}{2}}(\R^n) \to X^{\alpha} (\R_+^{n+1})}$  is an isometry, that is, for any $ u \in H^{\frac{\alpha}{2}}(\R^n),$ we have
\begin{equation} \label{extension norm}
\|E_{\alpha}(u)\|_{X^{\alpha} (\R^{n+1}_+)} = \|u \|_{H^{\frac{\alpha}{2}}(\R^n)}=\| (-\Delta)^{\frac{\alpha}{4}} u \|_{L^2(\R^n)}.
\end{equation}

Conversely,  for a function $w \in X^{\alpha} (\R_+^{n+1}),$ we denote its trace on $\R^n \times  \{y = 0\}$ as Tr$(w):=w(.,0)$. This trace operator is also well defined and satisfies

\begin{equation}\label{trace inequality between extension norm and fractional sobolev}
\|w(.,0) \|_{H^{\frac{\alpha}{2}}(\R^n)} \le \|w\|_{X^{\alpha} (\R_+^{n+1})}. 
\end{equation} 
We shall frequently use the following useful fact:  
Since $\alpha \in (0, 2)$, the weight $y^{1-\alpha}$ belongs to the Muckenhoupt class $A
_2$; \cite{Muckenhoupt}, which consists of all non-negative functions $w$ on $\R^n$ satisfying   for some constant C, the estimate 
 \begin{equation}
 \sup\limits_B (\frac{1}{|B|}\int_B w dx)(\frac{1}{|B|}\int_B w^{-1} dx) \le C,
 \end{equation} 
where the supremum is taken over all balls $B$ in $\R^n.$

If $\Omega \subset \R^{n+1}$ is an open domain, we denote by $L^2(\Omega, |y|^{1-\alpha})$ the space
of all measurable functions  on $\Omega$ such that $\|w\|^2_{L^2(\Omega , |y|^{1-\alpha})} = \int_{\Omega} |y|^{1-\alpha} |w|^2 dxdy < \infty$, and by $H^1(\Omega, |y|^{1-\alpha})$ the weighted Sobolev space $$H^1(\Omega, |y|^{1-\alpha}) = \left\{ w \in L^2(\Omega, |y|^{1-\alpha}): \, \nabla w \in  L^2(\Omega, |y|^{1-\alpha}) \right\}.$$ 
It is remarkable that most of the properties of classical Sobolev spaces, including the embedding theorems  have a weighted counterpart as long as the weight is in the Muckenhoupt class $A_2$ see \cite{Fabes-Kenig-Serapioni} and \cite{Gol'dshtein-Ukhlov}. 

 Note 
 that $H^1(\R^{n+1}_+ , y^{1-\alpha})$ -- up to a normalization factor--  is  also isometric to $X^{\alpha}(\R^{n+1}_+).$  In \cite{Caffarelli-Silvestre}, Caffarelli and Silvestre showed that the extension function $E_{\alpha}(u)$ is related to the fractional Laplacian of the original function $u$ in the following way:

\begin{equation*}
(-\Delta)^{\frac{\alpha}{2}}u(x)=\frac{\partial w}{\partial \nu^{\alpha}}:= - k_{\alpha} \lim\limits_{y \to 0^+} y^{1-\alpha} \frac{\partial w}{\partial y}(x,y).
\end{equation*}
With this representation, the non-local problem (\ref{Main problem}) can then be written as the following local problem:

\begin{equation} \label{Main problem.prime} 
\left\{\begin{array}{rll}
 - {\rm div}\,(y^{1-\alpha}\nabla  w)=0 \hfill  & \text{in} \ \mathbb{R}^{n+1}_+ \\
\frac{\partial w}{\partial \nu^{\alpha}}= \gamma \frac{w(.,0)}{|x|^{\alpha}} +w(.,0)^{2^*_\alpha-1}+ \frac{w(.,0)^{{2_{\alpha}^*(s)}-1}}{|x|^s}&\text{on } \ \mathbb{R}^n.
\end{array}\right.
\end{equation}
A function $w \in X^{\alpha}(\R_+^{n+1}) $ is said to be a weak solution to (\ref{Main problem.prime}), if for all $\varphi \in X^{\alpha}(\R_+^{n+1}),$ 
\begin{eqnarray*}
  k_{\alpha} \int_{\mathbb{R}_+^{n+1}} y^{1-\alpha} \langle \nabla w, \nabla \varphi \rangle  dxdy& =& \int_{\R^n} \gamma \frac{w(x,0)}{|x|^{\alpha}} \varphi dx  + \int_{\R^n} |w(x,0)|^{2^*_\alpha-2} w(x,0)\varphi dx\\
  &&+ \int_{\R^n} \frac{|w(x,0)|^{{2_{\alpha}^*(s)}-2}w(x,0)}{|x|^s} \varphi dx.
  \end{eqnarray*}

Note that for any weak solution $w$ in $X^{\alpha}(\R_+^{n+1})$ to  (\ref{Main problem.prime}), the function $u=w(.,0) $  defined in the sense of traces, is in $H^{\frac{\alpha}{2}}(\R^n)$ and is a weak solution to problem (\ref{Main problem}). 

The energy functional corresponding to (\ref{Main problem.prime}) is 

\begin{equation*}
\Phi(w)= \frac{1}{2} \| w\|^2_{X^{\alpha} (\R_+^{n+1})} - \frac{\gamma}{2}\int_{\mathbb{R}^n}\frac{|w(x, 0)|^{2}}{|x|^{\alpha}} dx -\frac{1}{2_{\alpha}^*} \int_{\R^n} |w(x, 0)|^{2_{\alpha}^*}\, dx -\frac{1}{2_{\alpha}^*(s)}\int_{\R^n} \frac{|w(x, 0)|^{2_{\alpha}^*(s)}}{|x|^{s}}  dx.
\end{equation*}

Hence the associated trace of any critical point $w$ of $\Phi$  in $X^{\alpha}(\R_+^{n+1})$  is a weak solution  for  (\ref{Main problem}).

The starting point of the  study of existence of weak solutions of the above problems  is therefore the following fractional trace inequalities which will guarantee that the above functionals are well defined and bounded below on the right function spaces. 
We start with the fractional Sobolev inequality \cite{Cotsiolis-Tavoularis}, which asserts that for $n > \alpha$ and $ 0<\alpha<2$, there exists a constant $C(n,\alpha) >0 $ such that 
\begin{equation}\label{fractional sobolev inequality}
\hbox{$ ( \int_{\R^n} |u|^{2_\alpha^*} dx )^{\frac{2}{2_\alpha^*}} \leq C(n,\alpha) \int_{\R^n} |({-}{ \Delta})^{\frac{\alpha}{4}}u|^2 dx$ \quad for all $u \in H^{\frac{\alpha}{2}} (\R^n),$
}
\end{equation}
where $2_{\alpha}^* = \frac{2n}{n-\alpha}.$ Another important inequality  is the fractional  Hardy inequality (see \cite{Frank-Lieb-Seiringer} and \cite{Herbst}), which states that under the same conditions on $n$ and $\alpha$, we have 
\begin{equation}\label{fractional Hardy inequality}
\hbox{$ \gamma_H \int_{\R^n}\frac{|u|^2}{|x|^{\alpha}} dx \leq \int_{\R^n} |({-}{ \Delta})^{\frac{\alpha}{4}}u|^2 dx$ \quad for all $u \in H^{\frac{\alpha}{2}} (\mathbb{R}^n)$,}
\end{equation}
where $\gamma_H$
is the best constant in the above inequality on $\R^n,$ that is
\begin{equation} 
\gamma_H=\gamma_H(\alpha):=\inf\left\{\frac{ \int_{\R^n} |({-}{ \Delta})^{\frac{\alpha}{4}}u|^2 dx}{ \int_{\R^n}\frac{|u|^2}{|x|^{\alpha}} dx}; \,\,  u \in H^{\frac{\alpha}{2}} (\R^n) \setminus \{0\}\right\}.
\end{equation}
It has also been shown there that $\gamma_H (\alpha)= 2^\alpha \frac{\Gamma^2(\frac{n+\alpha}{4})}{\Gamma^2(\frac{n-\alpha}{4})}$. 
Note that $\gamma_H(\alpha)$ converges to the best classical Hardy constant $\gamma_H(2)=\frac{(n-2)^2}{4}$ when ${\alpha \to 2}$. 

By interpolating these inequalities via H\"older's inequalities, one gets the following fractional Hardy-Sobolev inequalities.  

\begin{lemma}[Fractional Hardy-Sobolev Inequalities] Assume that $0<\alpha<2$,   and $ 0 \le s \le \alpha <n$. Then, there exist positive constants $ c \text{ and }  C,$ such that  
\begin{equation} \label{Fractional H-S inequality}
(\int_{\mathbb{R}^n} \frac{|u|^{2_{\alpha}^*(s)}}{|x|^{s}}dx)^\frac{2}{{2_{\alpha}^*(s)}} \leq c \int_{\R^n} |({-}{ \Delta})^{\frac{\alpha}{4}}u|^2 dx \quad \hbox{for all $u \in H^{\frac{\alpha}{2}} (\mathbb{R}^n).$} 
\end{equation}
Moreover, if $\gamma < \gamma_H:=2^\alpha \frac{\Gamma^2(\frac{n+\alpha}{4})}{\Gamma^2(\frac{n-\alpha}{4})}$, then 
\begin{equation} \label{fractional H-S-M inequality}
C(\int_{\mathbb{R}^n} \frac{|u|^{2_{\alpha}^*(s)}}{|x|^{s}}dx)^\frac{2}{{2_{\alpha}^*(s)}} \leq \int_{\R^n} |({-}{ \Delta})^{\frac{\alpha}{4}}u|^2 dx  - \gamma \int_{\mathbb{R}^n} \frac{|u|^{2}}{|x|^{\alpha}}dx \quad \hbox{for all $u \in H^{\frac{\alpha}{2}} (\mathbb{R}^n).$} 
\end{equation}
 
 \end{lemma}
 \begin{proof}
 Note that for s = 0 (resp., s = $\alpha$) the first inequality  is just the fractional Sobolev (resp., the fractional Hardy) inequality. We therefore have to only consider the case where $0 < s< \alpha$ in which case $2_{\alpha}^*(s)>2$. By applying H\"older's inequality, then the fractional Hardy and the fractional Sobolev inequalities, we have
\begin{align*}
\int_{\mathbb{R}^n} \frac{|u|^{2_{\alpha}^*(s)}}{|x|^{s}}dx &= \int_{\mathbb{R}^n} \frac{|u|^\frac{2s}{\alpha}}{|x|^{s}} |u|^{2_{\alpha}^*(s)-\frac{2s}{\alpha}}dx \\
& \le  (\int_{\mathbb{R}^n} \frac{ |u|^2}{|x|^{\alpha}}dx)^\frac{s}{\alpha} (\int_{\mathbb{R}^n}  |u|^{(2_{\alpha}^*(s) - \frac{2s}{\alpha}) \frac{\alpha}{\alpha-s} }dx)^\frac{\alpha-s}{\alpha} \\
&=  (\int_{\mathbb{R}^n} \frac{ |u|^2}{|x|^{\alpha}}dx)^\frac{s}{\alpha} (\int_{\mathbb{R}^n}  |u|^{2_{\alpha}^*}dx)^\frac{\alpha-s}{\alpha}  \\ 
&\le C_1 (\int_{\R^n} |({-}{ \Delta})^{\frac{\alpha}{4}}u|^2 dx)^\frac{s}{\alpha}  C_2 (\int_{\R^n} |({-}{ \Delta})^{\frac{\alpha}{4}}u|^2 dx)^{\frac{2_{\alpha}^*}{2}.\frac{\alpha-s}{\alpha}}\\
&\le c (\int_{\R^n} |({-}{ \Delta})^{\frac{\alpha}{4}}u|^2 dx)^\frac{n-s}{n-\alpha} = c (\int_{\R^n} |({-}{ \Delta})^{\frac{\alpha}{4}}u|^2 dx)^ \frac{2_{\alpha}^*(s)}{2}.
\end{align*}

From the definition of ${\gamma_H}$, it follows that for all $u \in H^{\frac{\alpha}{2}}(\R^n),$
$$\frac{\int_{\R^n} |({-}{ \Delta})^{\frac{\alpha}{4}}u|^2 dx-\gamma \int_{\mathbb{R}^n} \frac{|u|^{2}}{|x|^{\alpha}}dx}{ (\int_ {\mathbb{R}^n} \frac{|u|^  {2_{\alpha}^*(s)}}{|x|^s}dx)^\frac{2}{2_{\alpha}^*(s)} } \geq (1- \frac{\gamma}{\gamma_H}) \frac{ \int_{\R^n} |({-}{ \Delta})^{\frac{\alpha}{4}}u|^2 dx}{ (\int_ {\mathbb{R}^n} \frac{|u|^  {2_{\alpha}^*(s)}}{|x|^s}dx)^\frac{2}{2_{\alpha}^*(s)} }.$$
Hence (\ref{Fractional H-S inequality})  implies (\ref{fractional H-S-M inequality}) whenever $\gamma < \gamma_H.$

\end{proof}
\begin{remark}
One can use (\ref{extension norm}) to rewrite inequalities (\ref{fractional Hardy inequality}), (\ref{Fractional H-S inequality})  and  (\ref{fractional H-S-M inequality}) as the following trace class inequalities:

\begin{equation} \label{fractional Trace Hardy inequality}
\gamma_H \int_{\mathbb{R}^n} \frac{|w(x,0)|^{2}}{|x|^{\alpha}} \ dx \leq \ \| w\|^2_{X^{\alpha} (\R_+^{n+1})},
\end{equation}

\begin{equation} \label{fractional Trace H-S inequality}
(\int_{\mathbb{R}^n} \frac{|w(x,0)|^{2_{\alpha}^*(s)}}{|x|^{s}} \ dx)^\frac{2}{{2_{\alpha}^*(s)}} \leq c \ \| w\|^2_{X^{\alpha} (\R_+^{n+1})},
\end{equation}
\begin{equation} \label{fractional Trace H-S-M inequality }
C(\int_{\mathbb{R}^n} \frac{|w(x,0)|^{2_{\alpha}^*(s)}}{|x|^{s}} \ dx)^\frac{2}{{2_{\alpha}^*(s)}} \leq  \ \| w\|^2_{X^{\alpha} (\R_+^{n+1})} - \gamma\int_{\mathbb{R}^n}\frac{|w(x,0)|^{2}}{|x|^{\alpha}}dx. 
\end{equation}

\end{remark}

The best constant $\mu_{\gamma,s}(\R^n)$ in inequality (\ref{fractional H-S-M inequality}), can also be written as:

\begin{equation*}
S(n,\alpha,\gamma,s)= \inf\limits_{w \in X^{\alpha} (\R_+^{n+1})\setminus \{0\}} \frac{k_{\alpha}\int_{\R_+^{n+1}} y^{1-\alpha} |\nabla w|^2  dxdy - \gamma \int_{\R^n} \frac{|w(x,0)|^2}{|x|^{\alpha}} dx}{(\int_{\R^n} \frac{|w(x, 0)|^{2_{\alpha}^*(s)}}{|x|^{s}}dx)^\frac{2}{2_{\alpha}^*(s)}}.
\end{equation*}

We shall therefore investigate whether there exist extremal functions where this best constant is attained. Theorems \ref{Theorem the best fractional H-S constan} and \ref{Theorem Main result} can therefore be stated in the following way:

\begin{theorem} \label{Theorem Existence for fractional H-S-M, using the best constant}
Suppose  $0<\alpha<2$, $ 0 \le s < \alpha <n$,  
and $\gamma < \gamma_H $. 
We then have the following:
\begin{enumerate}
\item If $ \{ s > 0 \} \text{ or } \{ s=0 \text{ and } \gamma \ge 0 \}$, then $S(n,\alpha,\gamma,s)$ is attained in $X^{\alpha} (\R_+^{n+1})$. 

\item If $s=0$ and $\gamma < 0$, then there are no extremals for $S(n,\alpha,\gamma,s)$ in $X^{\alpha} (\R_+^{n+1})$.
 \end{enumerate} 
\end{theorem}

\begin{theorem}\label{Theorem Main result in extended form}
Let $0<\alpha<2,$  
$ 0 < s < \alpha<n$ and $ 0\le \gamma < \gamma_H.$  Then,  there exists a non-trivial weak solution to (\ref{Main problem.prime}) in $ X^\alpha (\R_+^{n+1}) $.
\end{theorem}

\section{Proof of Theorem \ref{Theorem the best fractional H-S constan}} \label{Section: the proof of attainability of S(n,alpha,gamma,s)}
We shall minimize the functional 
$$I_{\gamma,s}(w)= \frac{k_{\alpha}\int_{\R_+^{n+1}} y^{1-\alpha} |\nabla w|^2  dxdy - \gamma \int_{\R^n} \frac{|w(x,0)|^2}{|x|^{\alpha}}dx }{(\int_{\R^n} \frac{|w(x,0)|^{2_{\alpha}^*(s)}}{|x|^{s}}dx)^\frac{2}{2_{\alpha}^*(s)}}
$$ on the space $X^{\alpha} (\R_+^{n+1})$. Whenever $S(n,\alpha,\gamma,s)$ is attained at some $w\in X^{\alpha} (\R_+^{n+1})$, then it is clear that $u = \text{Tr} (w):= w(.,0)$ will be a function in $ H^{\frac{\alpha}{2}}(\R^n)$, where $\mu_{\gamma,s}(\R^n)$ is attained. 

Note first that inequality (\ref{fractional Trace Hardy inequality}) asserts that $X^{\alpha} (\R_+^{n+1})$ is embedded in the weighted space $L^2(\R^n, |x|^{-\alpha})$ and that this embeding is continuous. If $\gamma < \gamma_H$, it follows from (\ref{fractional Trace Hardy inequality}) that 
$$ \|w\|  := \left( k_{\alpha}\int_{\R_+^{n+1}} y^{1-\alpha} |\nabla w|^2  dxdy - \gamma \int_{\R^n} \frac{|w(x,0)|^2}{|x|^{\alpha}} dx \right)^\frac{1}{2} $$

is well-defined on $X^{\alpha} (\R_+^{n+1})$. Set $\gamma_+ = \text{max} \{\gamma,0\}$ and $\gamma_- = - \text{max} \{\gamma,0\}$. The following inequalities then hold for any $u \in X^{\alpha} (\R_+^{n+1})$,  

\begin{equation}\label{comparable norms}
(1-\frac{\gamma_+}{\gamma_H}) \|w\|^2_{X^{\alpha} (\R_+^{n+1})} \le \|w\|^2 \le (1+\frac{\gamma_-}{\gamma_H}) \|w\|^2_{X^{\alpha} (\R_+^{n+1})} .
\end{equation}
Thus, $\| \ . \ \|$ is equivalent to the norm $\| \ . \ \|_{X^{\alpha} (\R_+^{n+1})}$.

We start by considering the case when $s > 0$. Ekeland's variational principle \cite{Ekeland} applied to the functional $I(w):= I_{\gamma,s}(w)$ yields 
the existence of a minimizing sequence $(w_k)_k$ for $S(n,\alpha,\gamma,s)$ such that  as  $k \to \infty$, 
\begin{equation}
\int_ {\mathbb{R}^n} \frac{|w_k(x,0)|^  {2_{\alpha}^*(s)}}{|x|^s}dx=1,
\end{equation} 
\begin{equation}
I(w_k) \longrightarrow S(n,\alpha,\gamma,s),  
\end{equation}
and  
\begin{equation}
{I'(w_k) \to 0} \text{ in } (X^{\alpha} (\R_+^{n+1}))', 
\end{equation}
where   $(X^{\alpha} (\R_+^{n+1}))'$ denotes the dual of $ X^{\alpha} (\R_+^{n+1})$. Consider the  functionals $J,K:X^{\alpha} (\R_+^{n+1}) \longrightarrow \R $ by 
 $$J(w):= \frac{1}{2} \|w\|^2 = \frac{k_{\alpha}}{2}  \int_{\R_+^{n+1}} y^{1-\alpha} |\nabla w|^2  dxdy - \frac{\gamma}{2} \int_{\mathbb{R}^n} \frac{|w(x,0)|^{2}}{|x|^{\alpha}}dx,$$
 and
$$K(w):= \frac{1}{2_{\alpha}^*(s)}\int_{\R^n} \frac{|w(x,0)|^{2_{\alpha}^*(s)}}{|x|^{s}}  dx. $$
Straightforward computations yield that as $k \to \infty$, 
\begin{equation*}
J(w_k) \longrightarrow \frac{1}{2} S(n,\alpha,\gamma,s), 
\end{equation*}
and
\begin{equation}\label{Ekeland principle}
J'(w_k) - S(n,\alpha,\gamma,s) K'(w_k) \longrightarrow 0 \text{ in } (X^{\alpha} (\R_+^{n+1}))'. %
\end{equation}
Consider now the Levy concentration functions $Q$ of $\frac{|w_k(x,0)|^  {2_{\alpha}^*(s)}}{|x|^s}$, defined as
\begin{equation*}
Q(r)= \int_{B_r} \frac{|w_k(x,0)|^  {2_{\alpha}^*(s)}}{|x|^s} dx \quad \text{for} \quad r> 0,
\end{equation*}
where $B_r$ is the ball of radius $r$ in $\mathbb{R}^n$. Since $\int_ {\mathbb{R}^n} \frac{|w_k(x,0)|^  {2_{\alpha}^*(s)}}{|x|^s}dx=1$ for all $k \in \mathbb{N}$, then by continuity, and up to considering a subsequence, there exists $r_k>0$ such that 
\begin{equation*}
Q(r_k)= \int_{B_{r_k}} \frac{|w_k(x,0)|^  {2_{\alpha}^*(s)}}{|x|^s} dx = \frac{1}{2} \quad \hbox{for all $k \in \mathbb{N}$. }
\end{equation*}
Define the rescaled sequence $v_k(x,y):= r_k^{\frac{n-\alpha}{2}} w_k(r_k x, r_k y)$ for $k \in \mathbb{N}$ and $(x,y) \in \R_+^{n+1}$, in such a way that 
  $(v_k)_{k \in \mathbb{N}}$ is also a minimizing sequence for $S(n,\alpha,\gamma,s)$. Indeed, it is easy to check that $v_k \in X^{\alpha} (\R_+^{n+1})$ and that
\begin{equation*}
k_{\alpha}\int_{\R_+^{n+1}} y^{1-\alpha} |\nabla v_k|^2  dxdy - \gamma \int_{\mathbb{R}^n} \frac{|v_k(x,0)|^{2}}{|x|^{\alpha}}dx = k_{\alpha} \int_{\R_+^{n+1}} y^{1-\alpha} |\nabla w_k|^2  dxdy - \gamma \int_{\mathbb{R}^n} \frac{|w_k(x,0)|^{2}}{|x|^{\alpha}}dx,
\end{equation*}

\begin{equation}
\lim\limits_{k \to \infty} \left( k_{\alpha} \int_{\R_+^{n+1}} y^{1-\alpha} |\nabla v_k|^2  dxdy - \gamma \int_{\mathbb{R}^n} \frac{|v_k(x,0)|^{2}}{|x|^{\alpha}}dx \right)= S(n,\alpha,\gamma,s)
\end{equation}
and 
\begin{equation*}
\int_ {\mathbb{R}^n} \frac{|v_k(x,0)|^  {2_{\alpha}^*(s)}}{|x|^s}dx=\int_ {\mathbb{R}^n} \frac{|w_k(x,0)|^  {2_{\alpha}^*(s)}}{|x|^s}dx=1.
\end{equation*}
Moreover, we have that

\begin{equation} \label{Levy-type for v_k}
\int_{B_1} \frac{|v_k(x,0)|^  {2_{\alpha}^*(s)}}{|x|^s} dx = \frac{1}{2} \quad \hbox{for all $k \in \mathbb{N}$. }
\end{equation}

In addition, $\|v_k\|^2 = S(n,\alpha,\gamma,s) +o(1)$ as ${k \to \infty}$, so (\ref{comparable norms}) yields that $\left(\|v_k\|_{X^{\alpha} (\R_+^{n+1})} \right) _{k \in \mathbb{N}}$ is bounded.
 Therefore, without loss of generality, there exists a subsequence -still denoted $v_k$- such that
 
  \begin{equation} \label{extract weak and strong limit of minimizing sequence - ekeland} 
\hbox{$ v_k  \rightharpoonup v   \text{ in }   X^{\alpha}(\R_+^{n+1})$ \ and \ $
{v_k(.,0) \to v(.,0)}   \text{ in }  L_{loc}^{q}(\R^n), \ \text{for every }  1\le q < 2_{\alpha}^*.$}
\end{equation}
We shall show that the weak limit of the minimizing sequence is not identically zero, that is $v\not\equiv0$. 

Indeed, suppose $v\equiv0.$ It follows from (\ref{extract weak and strong limit of minimizing sequence - ekeland}) that 

\begin{equation} \label{weakly and strongly convergence to zero}
\hbox{$ v_k  \rightharpoonup 0   \text{ in }   X^{\alpha}(\R_+^{n+1})$ \ and \ $
{v_k(.,0) \to 0}   \text{ in }  L_{loc}^{q}(\R^n), \ \text{for every }  1\le q < 2_{\alpha}^*.$}
\end{equation}
For $\delta>0$, define $B_\delta^+:= \{(x,y) \in \R_+^{n+1}: |(x,y)| < \delta \}$, $B_\delta:= \{x \in \R^n: |x| < \delta\}$ and let $\eta \in C_0^{\infty}(\R_+^{n+1})$ be a cut-off function such that $\eta\equiv 1$  in $B^+_{\frac{1}{2}}$ and $0 \le \eta \le 1$ in $\R_+^{n+1}.$  

We use $\eta^2 v_k$ as test function in (\ref{Ekeland principle}) to get that
 
\begin{equation}\label{Use test eta^2 v_k in functional}
\begin{aligned}
&k_{\alpha} \int_{\R_+^{n+1}} y^{1-\alpha} \nabla v_k . \nabla (\eta^2 v_k ) dxdy - \gamma \int_{\mathbb{R}^n} \frac{v_k(x,0) (\eta^2 v_k(x,0) ) }{|x|^{\alpha}}dx\\
& \quad \quad   = S(n,\alpha,\gamma,s) \int_ {\mathbb{R}^n} \frac{|v_k(x,0)|^{2_{\alpha}^*(s)-1} (\eta^2 v_k(x,0) ) }{|x|^s}dx+o(1).
\end{aligned}
\end{equation}
Simple computations yield 
$ | \nabla(\eta v_k)|^2= |v_k \nabla \eta|^2 + \nabla v_k . \nabla(\eta^2  v_k),$
so that we have 
\begin{align*}
  & k_\alpha \int_{\R_+^{n+1}} y^{1-\alpha} | \nabla(\eta v_k)|^2 dxdy - k_\alpha \int_{\R_+^{n+1}} y^{1-\alpha} \nabla v_k . \nabla(\eta^2  v_k) dxdy\\
  & \quad \quad  = k_\alpha \int_{\R_+^{n+1}} y^{1-\alpha} |v_k \nabla \eta|^2 dxdy = k_\alpha \int_{E} y^{1-\alpha}|\nabla \eta|^2  |v_k |^2  dxdy,
\end{align*} 
where $E:= \text{Supp}(|\nabla \eta|).$ Since $\alpha\in (0,2)$, $y^{1-\alpha}$ is an $A_2$-weight, and since $E$ is bounded, we have that the embedding  $H^1(E, y^{1-\alpha}) \hookrightarrow L^2(E, y^{1-\alpha})$ 
 is compact (See \cite{B-C-D-S 1} and \cite{Gol'dshtein-Ukhlov}). It follows 
 from $(\ref{weakly and strongly convergence to zero})_1$ that
  $${ k_\alpha \int_{E} y^{1-\alpha} |v_k \nabla \eta|^2 dxdy \to 0} \text{ as } {k \to \infty  }.$$
Therefore, 
$$k_\alpha \int_{\R_+^{n+1}} y^{1-\alpha} | \nabla(\eta v_k)|^2 dxdy = k_\alpha \int_{\R_+^{n+1}} y^{1-\alpha} \nabla v_k . \nabla(\eta^2  w_k) dxdy + o(1).$$
 By plugging the above estimate into (\ref{Use test eta^2 v_k in functional}), and using (\ref{Levy-type for v_k}), we get that

\begin{equation} \label{estimate for grad term by test function eta^2 v_k}
\begin{aligned}
\|\eta v_k\|^2 & = k_{\alpha} \int_{\R_+^{n+1}} y^{1-\alpha} |\nabla (\eta v_k)|^2 dxdy - \gamma \int_{\mathbb{R}^n} \frac{|\eta v_k(x,0)|^2 }{|x|^{\alpha}}dx\\ 
& = S(n,\alpha,\gamma,s) \int_ {\mathbb{R}^n} \frac{|v_k(x,0)|^{2_{\alpha}^*(s)-2} (|\eta v_k(x,0)|^2) }{|x|^s}dx+o(1) \\ 
&    \le S(n,\alpha,\gamma,s) \int_{B_1} \frac{|v_k(x,0)|^{2_{\alpha}^*(s)}}{|x|^s} dx + o(1) \\ 
&= \frac{S(n,\alpha,\gamma,s)}{2^{1-\frac{2}{2_{\alpha}^*(s)}}} \left( \int_{B_1} \frac{|v_k(x,0)|^  {2_{\alpha}^*(s)}}{|x|^s} dx \right)^\frac{2}{2_{\alpha}^*(s)}+o(1). 
\end{aligned}
\end{equation}
By straightforward computations and H\"older's inequality, we get that
\begin{align*}
\left( \int_{B_1} \frac{|v_k(x,0)|^  {2_{\alpha}^*(s)}}{|x|^s} dx \right)^\frac{1}{2_{\alpha}^*(s)} & = \left( \int_{B_1} \frac{|\eta v_k(x,0) +(1- \eta) v_k(x,0) |^  {2_{\alpha}^*(s)}}{|x|^s} dx \right)^\frac{1}{2_{\alpha}^*(s)}\\
 &\le \left( \int_{B_1} \frac{|\eta v_k(x,0)|^  {2_{\alpha}^*(s)}}{|x|^s} dx \right)^\frac{1}{2_{\alpha}^*(s)} + \left( \int_{B_1} \frac{|(1- \eta )v_k(x,0)|^  {2_{\alpha}^*(s)}}{|x|^s} dx \right)^\frac{1}{2_{\alpha}^*(s)}\\
& \le \left( \int_{\R^n} \frac{|\eta v_k(x,0)|^  {2_{\alpha}^*(s)}}{|x|^s} dx \right)^\frac{1}{2_{\alpha}^*(s)} + C \left( \int_{B_1} |v_k(x,0)|^  {2_{\alpha}^*(s)} dx \right)^\frac{1}{2_{\alpha}^*(s)}.
\end{align*}
From $(\ref{weakly and strongly convergence to zero})_2$, and the fact that $2_{\alpha}^*(s) < 2_{\alpha}^*,$  we obtain
 $$ {\int_{B_1} |v_k(x,0)|^  {2_{\alpha}^*(s)} dx \to 0} \text{ as } {k \to \infty}. $$  
 Therefore, 
\begin{equation} \label{minkowski-type inequality for Hardy-Sobolev term}
\left( \int_{B_1} \frac{|v_k(x,0)|^  {2_{\alpha}^*(s)}}{|x|^s} dx \right)^\frac{2}{2_{\alpha}^*(s)} \le \left( \int_{\R^n} \frac{|\eta v_k(x,0)|^  {2_{\alpha}^*(s)}}{|x|^s} dx \right)^\frac{2}{2_{\alpha}^*(s)} + o(1).
\end{equation} 
Plugging the above inequality into (\ref{estimate for grad term by test function eta^2 v_k}), we get that 
\begin{align*}
\|\eta v_k\|^2 &= k_{\alpha} \int_{\R_+^{n+1}} y^{1-\alpha} |\nabla (\eta v_k)|^2 dxdy - \gamma \int_{\mathbb{R}^n} \frac{|\eta v_k(x,0)|^2 }{|x|^{\alpha}}dx \\
& \le \frac{S(n,\alpha,\gamma,s)}{2^{1-\frac{2}{2_{\alpha}^*(s)}}} \left(\int_ {\mathbb{R}^n} \frac{|\eta v_k(x,0)|^{2_{\alpha}^*(s)}}{|x|^s}dx\right)^\frac{2}{2_{\alpha}^*(s)}+o(1).
\end{align*}
On the other hand, it follows from the definition of $S(n,\alpha,\gamma,s)$ that 
$$ S(n,\alpha,\gamma,s) \left(\int_ {\mathbb{R}^n} \frac{|\eta v_k(x,0)|^{2_{\alpha}^*(s)}}{|x|^s}dx\right)^\frac{2}{2_{\alpha}^*(s)}  \le \|\eta v_k\|^2 \le  \frac{S(n,\alpha,\gamma,s)}{2^{1-\frac{2}{2_{\alpha}^*(s)}}} \left(\int_ {\mathbb{R}^n} \frac{|\eta v_k(x,0)|^{2_{\alpha}^*(s)}}{|x|^s}dx\right)^\frac{2}{2_{\alpha}^*(s)}+o(1).   $$
Note that $\frac{S(n,\alpha,\gamma,s)}{2^{1-\frac{2}{2_{\alpha}^*(s)}}} < S(n,\alpha,\gamma,s)$ for $ s \in (0, \alpha)$, hence (\ref{minkowski-type inequality for Hardy-Sobolev term}) yields that
$$o(1)=\int_ {\R^n} \frac{|\eta v_k(x,0)|^{2_{\alpha}^*(s)}}{|x|^s}dx = \int_{B_1} \frac{|v_k(x,0)|^  {2_{\alpha}^*(s)}}{|x|^s} dx + o(1).$$
This contradicts (\ref{Levy-type for v_k}) and therefore $v\not\equiv 0$. 

We now conclude by proving that $v_k$ converges weakly in $\R_+^{n+1}$ to $v$, and that  $\int_ {\R^n} \frac{|v(x,0)|^  {2_{\alpha}^*(s)}}{|x|^s}dx=1.$ Indeed, for $k \in \mathbb{N},$ let $\theta_k = v_k-v,$ and use the Brezis-Lieb Lemma (see \cite{Brezis-Lieb} and \cite{Yang}) to deduce that 
$$1=\int_{\R^n} \frac{|v_k(x,0)|^{2_{\alpha}^*(s)}}{|x|^{s}}dx = \int_{\R^n} \frac{|v(x,0)|^{2_{\alpha}^*(s)}}{|x|^{s}}dx + \int_{\R^n} \frac{|\theta_k(x,0)|^{2_{\alpha}^*(s)}}{|x|^{s}}dx+o(1),$$
which yields that both
\begin{equation}\label{Bound for H-S terms - v and theta_k. }
\hbox{$\int_{\R^n} \frac{|v(x,0)|^{2_{\alpha}^*(s)}}{|x|^{s}}dx$ and $ \int_{\R^n} \frac{|\theta_k(x,0)|^{2_{\alpha}^*(s)}}{|x|^{s}}dx$ are in the interval  $[0,1].$}
\end{equation}
The weak convergence $\theta_k \rightharpoonup 0$ in $X^\alpha(\R_+^{n+1})$ implies that 
$$\|v_k\|^2 = \|v+\theta_k\|^2 = \|v\|^2+\|\theta_k\|^2 + o(1).$$
By using  (\ref{Ekeland principle}) and the definition of  $S(n,\alpha,\gamma,s),$ we get that
\begin{equation*}
\begin{aligned}
o(1) &= \|v_k\|^2 - S(n,\alpha,\gamma,s)    \int_{\R^n} \frac{|v_k(x,0)|^{2_{\alpha}^*(s)}}{|x|^{s}}dx\\
& = \left(\|v\|^2 - S(n,\alpha,\gamma,s) \int_{\R^n} \frac{|v(x,0)|^{2_{\alpha}^*(s)}}{|x|^{s}}dx \right)       + \left(\|\theta_k\|^2 - S(n,\alpha,\gamma,s) \int_{\R^n} \frac{|\theta_k(x,0)|^{2_{\alpha}^*(s)}}{|x|^{s}}dx \right) + o(1) \\
& \ge  S(n,\alpha,\gamma,s) \left[ \left(\int_{\R^n} \frac{|v(x,0)|^{2_{\alpha}^*(s)}}{|x|^{s}}dx\right)^{\frac{2}{2^*_\alpha(s)}} - \int_{\R^n} \frac{|v(x,0)|^{2_{\alpha}^*(s)}}{|x|^{s}}dx \right] \\
&+ S(n,\alpha,\gamma,s) \left[ \left(\int_{\R^n} \frac{|\theta_k(x,0)|^{2_{\alpha}^*(s)}}{|x|^{s}}dx\right)^{\frac{2}{2^*_\alpha(s)}} - \int_{\R^n} \frac{|\theta_k(x,0)|^{2_{\alpha}^*(s)}}{|x|^{s}}dx \right]+o(1).
\end{aligned}  
\end{equation*}
Set now
$$A:=\left(\int_{\R^n} \frac{|v(x,0)|^{2_{\alpha}^*(s)}}{|x|^{s}}dx\right)^{\frac{2}{2^*_\alpha(s)}} - \int_{\R^n} \frac{|v(x,0)|^{2_{\alpha}^*(s)}}{|x|^{s}}dx,$$
and 
 $$B:= \left(\int_{\R^n} \frac{|\theta_k(x,0)|^{2_{\alpha}^*(s)}}{|x|^{s}}dx\right)^{\frac{2}{2^*_\alpha(s)}} - \int_{\R^n} \frac{|\theta_k(x,0)|^{2_{\alpha}^*(s)}}{|x|^{s}}dx.$$\\
Note that since $2_\alpha^*(s) > 2,$ we have $a^{\frac{2}{2_\alpha^*(s)}} \ge a $ for every $a \in [0,1]$,  and equality holds if and only if $a=0$ or $a=1.$ 
It then follows from (\ref{Bound for H-S terms - v and theta_k. }) that both  $A$ and $B$ are non-negative. On the other hand, the last inequality implies that $A+B =o(1),$ which means that $A=0 $ and $B=o(1)$, that is  
\begin{equation*} 
\int_{\R^n} \frac{|v(x,0)|^{2_{\alpha}^*(s)}}{|x|^{s}}dx = \left(\int_{\R^n} \frac{|v(x,0)|^{2_{\alpha}^*(s)}}{|x|^{s}}dx\right)^{\frac{2}{2^*_\alpha(s)}}, 
\end{equation*}
hence
$$ \text{ either } \int_{\R^n} \frac{|v(x,0)|^{2_{\alpha}^*(s)}}{|x|^{s}}dx =0 \text{ or }  \int_{\R^n} \frac{|v(x,0)|^{2_{\alpha}^*(s)}}{|x|^{s}}dx =1. $$
The fact that $v \not\equiv 0$ yields  $\int_{\R^n} \frac{|v(x,0)|^{2_{\alpha}^*(s)}}{|x|^{s}}dx \neq 0,$
and  $\int_{\R^n} \frac{|v(x,0)|^{2_{\alpha}^*(s)}}{|x|^{s}}dx =1,$ which yields that   
$$k_{\alpha} \int_{\R_+^{n+1}} y^{1-\alpha} |\nabla v|^2  dxdy - \gamma \int_{\mathbb{R}^n} \frac{|v(x,0)|^{2}}{|x|^{\alpha}}dx = S(n,\alpha,\gamma,s). $$
Without loss of generality we may assume  $v \ge 0$ (otherwise we take $|v|$ instead of $v$), and we then obtain a positive extremal for $S(n,\alpha,\gamma,s)$ in the case $s \in (0, \alpha).$\\

$\bullet$ Suppose now that  $s=0$ and $\gamma \ge 0$.
By a result in \cite{Cotsiolis-Tavoularis}, extremals exist for $S(n,\alpha,\gamma, s)$ whenever $s=0$ and $\gamma = 0$. Hence, we only need to show that there exists an extremal for 
$S(n,\alpha,\gamma,0)$ in the case $\gamma > 0$.
First note that in this case, we have that \begin{equation} S(n,\alpha,\gamma,0) < S(n,\alpha,0,0).
\end{equation}
 Indeed, if $w \in X^{\alpha}(\R_+^{n+1}) \setminus \{0\}$  is an extremal for $S(n,\alpha,0,0)$, then by  
 estimating the functional  at $w,$ and using the fact that $\gamma > 0$, we obtain
 \begin{equation*}
 \begin{aligned}
S(n,\alpha,\gamma,0)&=\inf\limits_{u \in  X^{\alpha}(\R_+^{n+1}) \setminus \{0\}}  \quad  \frac{\| u\|^2_{X^{\alpha}(\R_+^{n+1})}- \gamma\int_{\R^n}\frac{|u(x,0)|^{2}}{|x|^{\alpha}}dx}{(\int_{\R^n} |u(x,0)|^{2_{\alpha}^*}dx)^\frac{2}{2_{\alpha}^*}}\\
&  \le \frac{\| w\|^2_{X^{\alpha}(\R_+^{n+1})}- \gamma\int_{\R^n}\frac{|w(x,0)|^{2}}{|x|^{\alpha}}dx}{(\int_{\R^n} |w(x,0)|^{2_{\alpha}^*}dx)^\frac{2}{2_{\alpha}^*}}< \frac{\| w\|^2_{X^{\alpha}(\R_+^{n+1})}}{(\int_{\R^n} |w(x,0)|^{2_{\alpha}^*}dx)^\frac{2}{2_{\alpha}^*}}=S(n,\alpha,0,0). 
 \end{aligned}
\end{equation*}
Now we show that $S(n,\alpha,\gamma,0)$ is attained whenever $S(n,\alpha,\gamma,0) < S(n,\alpha,0,0).$  

 Indeed, let $(w_k)_{k\in \mathbb{N}} \subset X^{\alpha}(\R_+^{n+1}) \setminus \{ 0 \}$ be a minimizing sequence for $S(n,\alpha,\gamma,0)$. Up to multiplying by a positive constant, we assume that
 \begin{equation}
\lim\limits_{k \to \infty} \left( k_{\alpha} \int_{\R_+^{n+1}} y^{1-\alpha} |\nabla w_k|^2  dxdy - \gamma \int_{\mathbb{R}^n} \frac{|w_k(x,0)|^{2}}{|x|^{\alpha}}dx \right)= S(n,\alpha,\gamma,0)
\end{equation}
and
 \begin{equation} \label{minimizing sequence for sobolev is 1}
\int_ {\mathbb{R}^n} |w_k(x,0)|^  {2_{\alpha}^*}dx=1.
\end{equation}
 The sequence $\left(\|w_k\|_{X^{\alpha} (\R_+^{n+1})} \right) _{k \in \mathbb{N}}$ is therefore bounded, and there exists a subsequence - still denoted $w_k$- such that $w_k  \rightharpoonup w \text{ weakly }\text{in }   X^{\alpha} (\R_+^{n+1}).$ The weak convergence implies that 
\begin{align*} \label{Brezis-Lieb Lemma for extension-norm}
\begin{split}
\| w_k\|^2_{X^{\alpha}(\R_+^{n+1}) }&=  \| w_k - w\|^2_{X^{\alpha}(\R_+^{n+1})} +\|  w\|^2_{X^{\alpha}(\R_+^{n+1})}+ 2 k_\alpha \int_{\R_+^{n+1}} y^{1-\alpha} \langle \nabla w, \nabla (w-w_k) \rangle dx dy\\
& = \| w_k - w\|^2_{X^{\alpha}(\R_+^{n+1})} +\|  w\|^2_{X^{\alpha}(\R_+^{n+1})}+o(1)
\end{split}
\end{align*}
and 
\begin{align*}
\int_{\R^n} \frac{|w(x,0)|^{2}}{|x|^{\alpha}}dx &= 
\int_{\R^n} \frac{|(w-w_k)(x,0)|^{2}}{|x|^{\alpha}}dx +\int_{\R^n} \frac{|w_k(x,0)|^{2}}{|x|^{\alpha}}dx + 2\int_{\R^n} \frac{w_k(x,0) (w-w_k)(x,0)}{|x|^{\alpha}}dx   \\
& = \int_{\R^n} \frac{|(w-w_k)(x,0)|^{2}}{|x|^{\alpha}}dx +\int_{\R^n} \frac{|w_k(x,0)|^{2}}{|x|^{\alpha}}dx+ o(1).
\end{align*}
 The Brezis-Lieb Lemma (\cite[Theorem 1]{Brezis-Lieb}) and (\ref{minimizing sequence for sobolev is 1}) yield that 
$\int_ {\R^n} |(w_k - w) (x,0)|^  {2_{\alpha}^*}dx \le 1,$
for large $k$, hence  
\begin{align*}
S(n,\alpha,\gamma,0)  &=  \|w_k\|^2_{X^{\alpha} (\R_+^{n+1})}- \gamma \int_{\R^n} \frac{|w_k(x,0)|^{2}}{|x|^{\alpha}}dx+ o(1) \\
& \ge \| w_k - w \|^2_{X^{\alpha} (\R_+^{n+1})} +\| w\|^2_{X^{\alpha} (\R_+^{n+1})}- \gamma \int_{\R^n} \frac{|w(x,0)|^{2}}{|x|^{\alpha}}dx +o(1)\\
&\ge S(n,\alpha,0,0) (\int_ {\R^n} |(w_k - w) (x,0)|^  {2_{\alpha}^*}dx)^\frac{2}{2_{\alpha}^*}+  S(n,\alpha,\gamma,0)(\int_ {\R^n} |w(x,0)|^  {2_{\alpha}^*}dx)^\frac{2}{2_{\alpha}^*}+o(1)\\ 
&\ge S(n,\alpha,0,0) \int_ {\R^n} |(w_k - w) (x,0)|^  {2_{\alpha}^*}dx+ S(n,\alpha,\gamma,0)  \int_ {\R^n} |w (x,0)|^  {2_{\alpha}^*}dx+o(1).
\end{align*}
 Use the Brezis-Lieb Lemma again to get that
  \begin{align*}
S(n,\alpha,\gamma,0)  &\ge \left(S(n,\alpha,0,0) - S(n,\alpha,\gamma,0)\right) \int_ {\mathbb{R}^n} |(w_k - w) (x,0)|^  {2_{\alpha}^*}dx+S(n,\alpha,\gamma,0) \int_ {\mathbb{R}^n} |w_k  (x,0)|^  {2_{\alpha}^*}dx+o(1) \\ 
& =\left(S(n,\alpha,0,0) - S(n,\alpha,\gamma,0)\right) \int_ {\mathbb{R}^n} |(w_k - w) (x,0)|^  {2_{\alpha}^*}dx+S(n,\alpha,\gamma,0)+o(1) .
 \end{align*}
Since $S(n,\alpha,\gamma,0) < S(n,\alpha,0,0)$, we get that ${w_k(.,0) \to w(.,0)}$ in $L^{2_\alpha^*}(\R^n),$ that is
$ \int_ {\mathbb{R}^n} |w (x,0)|^  {2_{\alpha}^*}dx =1.$  The lower semi-continuity of $I$ then implies that $w$ is a minimizer for $S(n,\alpha,\gamma,0).$ Note that  $|w| $ is also an extremal in $X^\alpha(\R_+^{n+1})$ for $S(n,\alpha,\gamma,0),$ therefore there exists a non-negative extremal for $S(n,\alpha,\gamma,s)$ in the case $\gamma > 0$ and $s=0$, and this completes the proof of the case when $s=0$ and $\gamma \geq 0$.

Now we consider the case when $\gamma <0$. 

\begin{claim} \label{Claim: no extremal when gamma <0 }
 If $\gamma\le 0$, then  $S(n,\alpha,\gamma,0) = S(n,\alpha,0,0)$, hence, there are no extremals for $S(n,\alpha,\gamma, 0)$ whenever $\gamma<0.$ 
\end{claim}
Indeed, we first note that for $\gamma \le 0, $ we have 
$S(n,\alpha,\gamma,0) \ge S(n,\alpha,0,0).$
On the other hand, if we consider $w \in X^{\alpha}(\R_+^{n+1}) \setminus \{0\}$ to be an extremal for $S(n,\alpha,0,0)$ and define for $\delta \in \R $, and $\bar{x} \in \R^n$, the function  
$ w_{\delta}:= w(x-\delta \bar{x} , y)$ for $x\in\R^n$ and $y \in \R_+,$ then by a change of variable, we  get 
$$S(n, \alpha, \gamma, 0) \leq I_{\delta}: =\frac{ \| w_\delta\|^2_{X^{\alpha}(\R_+^{n+1})}- \gamma \int_{\R^n} \frac{ |w_\delta(x,0)|^2 }{|x|^\alpha}dx}{(\int_{\mathbb{R}^n} | w_\delta(x,0)|^{2_{\alpha}^*}dx)^\frac{2}{2_{\alpha}^*}} =  \frac{ \| w\|^2_{X^{\alpha}(\R_+^{n+1})}- \gamma \int_{\R^n} \frac{ |w(x,0)|^2 }{|x+\delta \bar{x}|^\alpha}dx}{(\int_{\mathbb{R}^n} | w(x,0)|^{2_{\alpha}^*}dx)^\frac{2}{2_{\alpha}^*}},$$
so that 
$$ S(n, \alpha, \gamma, 0) \leq \lim\limits_{\delta \to \infty} I_\delta = \frac{ \| w\|^2_{X^{\alpha}(\R_+^{n+1})}}{(\int_{\mathbb{R}^n} | w(x,0)|^{2_{\alpha}^*}dx)^\frac{2}{2_{\alpha}^*}}=S(n,\alpha,0,0).$$
Therefore,  $S(n,\alpha,\gamma, 0) = S(n,\alpha,0,0).$ 
Since there are extremals for $S(n,\alpha,0,0)$
(see \cite{Cotsiolis-Tavoularis}), there is none for     $S(n,\alpha,\gamma,0)$ whenever $\gamma<0.$
This establishes (2) 
and completes the proof of Theorem \ref{Theorem Existence for fractional H-S-M, using the best constant}.  

Back to Theorem \ref{Theorem the best fractional H-S constan}, since the $\alpha$-harmonic function $w \ge 0$ is a minimizer for $S(n,\alpha,\gamma,s)$ in $X^{\alpha}(\R_+^{n+1}) \setminus \{0\},$ which exists from Theorem \ref{Theorem Existence for fractional H-S-M, using the best constant}, then 
$ u:=\text{Tr}(w)= w(.,0) \in H^\frac{\alpha}{2}(\R^n)\setminus \{0\}$ and by  (\ref{extension norm}), $u$ is a minimizer for $\mu_{\gamma,s}(\R^n)$ in $H^\frac{\alpha}{2}(\R^n) \setminus \{0\}$. Therefore,  (1) and (2) of Theorem \ref{Theorem the best fractional H-S constan} hold.

For (3),
let $u^*$ be the Schwarz symmetrization of $u$. By the fractional Polya-Szeg\"o inequality \cite{Y.J. Park P-S inequality}, we have
$$\| (-\Delta)^{\frac{\alpha}{2}} u^* \|^2_{L^2(\R^n)} \le \| (-\Delta)^{\frac{\alpha}{2}} u \|^2_{L^2(\R^n)}. $$
Furthermore, it is clear (Theorem 3.4. of \cite{Lieb-Loss}) that  
 \begin{equation*}
\hbox{$    \int_{\R^n}\frac{|u|^{2}}{|x|^{\alpha}}dx \le \int_{\R^n}\frac{|u^*|^{2}}{|x|^{\alpha}}dx $\quad and \quad $
\int_{\R^n} \frac{|u|^{2_{\alpha}^*(s)}}{|x|^{s}}dx \le \int_{\R^n} \frac{|u^*|^{2_{\alpha}^*(s)}}{|x|^{s}}dx.$}
         \end{equation*}
Combining the above inequalities and the fact that $ \gamma \ge 0,$ we get that 
$$\mu_{\gamma,s}(\R^n) \le \frac{ \| (-\Delta)^{\frac{\alpha}{2}} u^* \|^2_{L^2(\R^n)}-
\gamma\int_{\R^n}\frac{|u^*|^{2}}{|x|^{\alpha}}dx}{(\int_{\R^n} \frac{|u^*|^{2_{\alpha}^*(s)}}{|x|^{s}}dx)^\frac{2}{2_{\alpha}^*(s)}} \le \frac{ \| (-\Delta)^{\frac{\alpha}{2}} u \|^2_{L^2(\R^n)}-
\gamma\int_{\R^n}\frac{|u|^{2}}{|x|^{\alpha}}dx}{(\int_{\R^n} \frac{|u|^{2_{\alpha}^*(s)}}{|x|^{s}}dx)^\frac{2}{2_{\alpha}^*(s)}}=\mu_{\gamma,s}(\R^n).  $$
This implies that $u^*$ is also a minimizer and achieves the infimum of $\mu_{\gamma,s}(\R^n).$ Therefore the equality sign holds in all the inequalities above, that is 
 \begin{equation*}
\hbox{$   \gamma \int_{\R^n}\frac{|u|^{2}}{|x|^{\alpha}}dx = \gamma \int_{\R^n}\frac{|u^*|^{2}}{|x|^{\alpha}}dx $\quad and \quad $
\int_{\R^n} \frac{|u|^{2_{\alpha}^*(s)}}{|x|^{s}}dx = \int_{\R^n} \frac{|u^*|^{2_{\alpha}^*(s)}}{|x|^{s}}dx.$}
         \end{equation*}
From Theorem 3.4. of \cite{Lieb-Loss}, in the case of equality, it follows that $u=|u|= u^*$ if either $\gamma \neq 0$ or if $s \neq 0$. In particular, $u$ is positive, radially symmetric and decreasing about origin. Hence $u$ must approach a limit as ${|x| \to \infty},$ which must be zero.

\section{Proof of Theorem \ref{Theorem Main result}} 

We shall now use the existence of extremals for the fractional Hardy-Sobolev type inequalities, established in Section \ref{Section: the proof of attainability of S(n,alpha,gamma,s)}, to prove that there exists a nontrivial  weak solution for (\ref{Main problem.prime}). 
The energy functional $\Psi$ associated to (\ref{Main problem.prime}) is defined as follows:
\begin{equation}\label{Psi}
\Psi(w)= \frac{1}{2} \| w\|^2 -\frac{1}{2_{\alpha}^*} \int_{\R^n} |u|^{2_{\alpha}^*}dx -\frac{1}{2_{\alpha}^*(s)}\int_{\R^n} \frac{|u|^{2_{\alpha}^*(s)}}{|x|^{s}}  dx, \quad \text{ for } w \in X^{\alpha} (\R_+^{n+1}),
\end{equation}
where again $u:= Tr (w)=w(.,0)$.  Fractional trace Hardy, Sobolev and Hardy-Sobolev inequalities  yield that     $\Psi \in C^1(X^{\alpha} (\R_+^{n+1})).$ Note that a  weak solution to (\ref{Main problem.prime}) is a nontrivial critical point of $\Psi$.

Throughout this section,  we use the following notation for any sequence $(w_k)_{k \in \mathbb{N}} \in X^{\alpha} (\R_+^{n+1})$:  $$u_k:=\text{Tr}(w_k)=w_k(.,0), \ \text{ for all } k \in \mathbb{N}.$$
We split the proof in three parts: 

\subsection{Existence of a suitable Palais-Smale sequence}
We first verify that the energy functional $\Psi$ satisfies the  conditions of  the Mountain Pass Lemma leading to a minimax energy level that is below a suitable threshold. The following is standard.

\begin{lemma}[Ambrosetti and Rabinowitz \cite{Ambrosetti-Rabinowitz}] \label{Theorem MPT- Ambrosetti-Rabinowitz version}
Let $(V,\|\,\|$) be a Banach space and $\Psi: {V \to \R}$ a $C^1-$functional satisfying the following conditions:

(a) $ \Psi(0)=0,$ \\
(b) There exist $\rho, R>0$ such that $\Psi(u) \ge \rho$ for all $u \in V$, with $\|u\|=R,$\\
(c) There exists $v_0 \in V $ such that $\limsup\limits_{t \to \infty} \Psi(tv_0) <0.$

Let $t_0>0$ be such that $\|t_0v_0\|> R$ and $\Psi(t_0v_0)<0,$ and define
$$c_{v_0}(\Psi):=\inf\limits_{\sigma \in \Gamma} \sup\limits_{t \in[0,1]} \Psi(\sigma(t)) \text{ where } \Gamma:= \{\sigma\in C([0,1],V): \sigma(0)=0 \text{ and } \sigma(1)=t_0v_0 \}.$$
Then, $c_w(\Psi) \ge \rho>0$, and there exists a Palais-Smale sequence at level $c_w(\Psi)$, that is  $(w_k)_{k \in \mathbb{N}} \in V$ such that 
 $$
 \lim\limits_{k \to \infty} \Psi(w_k)=c_{v_0}(\Psi) \ \text{ and } \lim\limits_{k \to \infty}  \Psi'(w_k)=0 \, \, \text{strongly in} \, V' .$$
\end{lemma}
We now prove the following.

\begin{proposition} \label{MPT with bound} Suppose  $0 \le \gamma < \gamma_H \text{ and } 0 \le s<\alpha$, and consider $\Psi$ defined in (\ref{Psi}) on the Banach space $X^{\alpha} (\R_+^{n+1})$. Then, there exists  $w \in X^{\alpha} (\R_+^{n+1}) \setminus \{0\}$ such that $w\ge 0$  and  $0<c_w(\Psi)<c^\star,$ where  
\begin{equation} \label{Definition of C^star}
c^\star = \text{min} \left\{ \frac{\alpha}{2n} S(n,\alpha,\gamma,0)^{\frac{n}{\alpha}}, \frac{\alpha-s}{2(n-s)} S(n,\alpha,\gamma,s)^{\frac{n-s}{\alpha-s}}   \right\},
\end{equation}
and a Palais-Smale sequence $(w_k)_{k \in \mathbb{N}}$ in $X^{\alpha} (\R_+^{n+1})$ at energy level $c_w(\Psi)$, that is,  
\begin{equation} \label{P-S condition (Lim) on minimizing sequence}
 \lim\limits_{k \to \infty}  \Psi'(w_k)=0 \text{ strongly in } (X^{\alpha} (\R_+^{n+1}))' \text{ and } \lim\limits_{k \to \infty} \Psi(w_k)= c_w(\Psi).
\end{equation}
\end{proposition}
\begin{proof}[Proof of Proposition \ref{MPT with bound}]
In the sequel, we will use freely the following elementary identities involving $2_{\alpha}^*(s)$: 
$$\hbox{$ \frac{1}{2} - \frac{1}{2_{\alpha}^*}= \frac{\alpha}{2n}$,\quad   
$\frac{2_{\alpha}^*}{2_{\alpha}^*-2}= \frac{n}{\alpha}$,\quad    $\frac{1}{2} - \frac{1}{2_{\alpha}^*(s)}= \frac{\alpha-s}{2(n-s)}$\quad  and \quad  $\frac{2_{\alpha}^*(s)}{2_{\alpha}^*(s)-2}=\frac{n-s}{\alpha-s}.$}$$
First, we note that the functional $\Psi$ satisfies the  hypotheses of Lemma \ref{Theorem MPT- Ambrosetti-Rabinowitz version}, and that condition (c) is satisfied for any $w \in X^{\alpha} (\R_+^{n+1}) \setminus \{0\}.$
Indeed, it is standard to show that $\Psi \in C^1(X^{\alpha} (\R_+^{n+1}))$ and clearly  $\Psi(0)=0$, so that (a) of Lemma \ref{Theorem MPT- Ambrosetti-Rabinowitz version} is satisfied. For (b) note that by the definition of $S(n,\alpha,\gamma,s),$ we have that 
$$ (\int_{\R^n} |u|^{2_{\alpha}^*}dx)^\frac{2}{2_{\alpha}^*}\le  S(n,\alpha,\gamma,0)^{-1} \|w\|^2
\text { and }  (\int_{\R^n} \frac{|u|^{2_{\alpha}^*(s)}}{|x|^{s}}dx)^\frac{2}{2_{\alpha}^*(s)}\le  S(n,\alpha,\gamma,s)^{-1} \|w\|^2.$$
Hence,
\begin{equation}
\begin{aligned}
\Psi(w) &\ge  \frac{1}{2} \| w\|^2-\frac{1}{2_{\alpha}^*} S(n,\alpha,\gamma,0)^{-\frac{2_{\alpha}^*}{2}} \| w\|^{2_{\alpha}^*} -\frac{1}{2_{\alpha}^*(s)} S(n,\alpha,\gamma,s)^{-\frac{2_{\alpha}^*(s)}{2}} \| w\|^{2_{\alpha}^*(s)} \\
& = \left( \frac{1}{2}-\frac{1}{2_{\alpha}^*} S(n,\alpha,\gamma,0)^{-\frac{2_{\alpha}^*}{2}} \| w\|^{2_{\alpha}^*-2} -\frac{1}{2_{\alpha}^*(s)} S(n,\alpha,\gamma,s)^{-\frac{2_{\alpha}^*(s)}{2}} \| w\|^{2_{\alpha}^*(s)-2} \right) \|w\|^2.
\end{aligned}
\end{equation}
Since $s \in [0,\alpha),$ we have that $2_{\alpha}^*-2 >0 $ and $2_{\alpha}^*(s)-2 >0.$ Thus, by (\ref{comparable norms}), we can find $R>0$ such that  $\Psi(w) \ge \rho$ for all $w \in X^{\alpha} (\R_+^{n+1})$ with $\|w\|_{X^{\alpha} (\R_+^{n+1})} = R.$  

Regarding (c), note that 
$$\Psi(tw) = \frac{t^2}{2} \| w\|^2 -\frac{t^{2_{\alpha}^*}}{2_{\alpha}^*} \int_{\R^n} |u|^{2_{\alpha}^*}dx -\frac{ t^{ 2_{\alpha}^*(s)}  }{2_{\alpha}^*(s)}\int_{\R^n} \frac{|u|^{2_{\alpha}^*(s)}}{|x|^{s}}  dx,$$
hence $ \lim\limits_{t \to \infty} \Psi(tw)=  -\infty$ for any $w \in X^{\alpha} (\R_+^{n+1}) \setminus \{0\}$, which means that    there exists $t_w>0$ such that $\|t_w w\|_{X^{\alpha} (\R_+^{n+1})} > R$ and $\Psi(tw) <0,$ for $t\ge t_w.$ 

Now we show that there exists  $ w \in X^{\alpha} (\R_+^{n+1}) \setminus \{0\}$ such that $w\ge 0$  and 
\begin{equation}\label{bound for c_w when s=0}
c_w (\Psi)< \frac{\alpha}{2n} S(n,\alpha,\gamma,0)^{\frac{n}{\alpha}}.
\end{equation}
From Theorem \ref{Theorem Existence for fractional H-S-M, using the best constant}, we know that there exists a  non-negative extremal $w$ in $X^{\alpha} (\R_+^{n+1})$ for $S(n,\alpha,\gamma,0)$ whenever $\gamma \ge 0.$  By the definition of $t_w$ and the fact that $c_w > 0,$ we obtain

$$c_w(\Psi) \le \sup\limits_{t\ge0} \Psi(tw)\le \sup\limits_{t\ge0}f(t),
\quad \text{ where }  f(t) = \frac{t^2}{2} \| w\|^2 -\frac{t^{2_{\alpha}^*}}{2_{\alpha}^*} \int_{\R^n} |u|^{2_{\alpha}^*}dx \quad \forall t>0. $$
Straightforward computations yield that $f(t)$ attains  its maximum at the point $\tilde{t} =  \left( \frac{\|w\|^2}{\int_{\R^n} |u|^{2_{\alpha}^*}dx} \right)^\frac{1}{2^*_\alpha -2}.$ It follows that 
$$\sup\limits_{t\ge0}f(t) = (\frac{1}{2} - \frac{1}{2_{\alpha}^*})  \left( \frac{\|w\|^2}{(\int_{\R^n} |u|^{2_{\alpha}^*}dx)^\frac{2}{2_{\alpha}^*}}\right)^\frac{2_{\alpha}^*}{2_{\alpha}^*-2} = \frac{\alpha}{2n} \left( \frac{\|w\|^2}{(\int_{\R^n} |u|^{2_{\alpha}^*}dx)^\frac{2}{2_{\alpha}^*}}\right)^{\frac{n}{\alpha}}.$$
Since $w$ is an extremal for $S(n,\alpha,\gamma,0)$, we get that 
$$c_w(\Psi) \le \sup\limits_{t\ge0}f(t) = \frac{\alpha}{2n} S(n,\alpha,\gamma,0)^{\frac{n}{\alpha}}. $$

We now need to show that equality does not hold in (\ref{bound for c_w when s=0}). Indeed, otherwise  
we would have that $0<c_w = \sup\limits_{t\ge0}\Psi(tw) = \sup\limits_{t\ge0}f(t).$ Consider  $t_1$ (resp. $t_2 > 0$) where $\sup\limits_{t\ge0} \Psi(tw) (\text{resp.,} \ \sup\limits_{t\ge0}f(t) )$ is attained. We get that  
$$f(t_1) - \frac{ t_1^{ 2_{\alpha}^*(s)}  }{2_{\alpha}^*(s)}\int_{\R^n} \frac{|w(x,0)|^{2_{\alpha}^*(s)}}{|x|^{s}}  dx = f(t_2),$$
which means that $f(t_1) > f(t_2)$ since  $t_1 > 0$. This contradicts   the fact that $t_2$ is a maximum point of $f(t)$, hence the strict inequality in (\ref{bound for c_w when s=0}).

To finish the proof of Proposition \ref{MPT with bound}, we can assume without loss that 
$$\frac{\alpha-s}{2(n-s)} S(n,\alpha,\gamma,s)^{\frac{n-s}{\alpha-s}} < \frac{\alpha}{2n} S(n,\alpha,\gamma,0)^{\frac{n}{\alpha}}.$$

Let now $w$ in $X^{\alpha} (\R_+^{n+1}) \setminus \{0\} $ be a positive minimizer for $S(n,\alpha,\gamma,s)$, whose existence was established in Section \ref{Section: the proof of attainability of S(n,alpha,gamma,s)}, and set  
$$\bar{f}(t) = \frac{t^2}{2} \| w\|^2  -\frac{ t^{ 2_{\alpha}^*(s)}  }{2_{\alpha}^*(s)}\int_{\R^n} \frac{|u|^{2_{\alpha}^*(s)}}{|x|^{s}}  dx.$$
As above, we have 
$$c_w(\Psi) \le \sup\limits_{t\ge0}f(t) =  (\frac{1}{2} - \frac{1}{2_{\alpha}^*(s)})  \left( \frac{\|w\|^2}{(\int_{\R^n} \frac{|u|^{2_{\alpha}^*(s)}}{|x|^{s}}dx)^\frac{2}{2_{\alpha}^*(s)}}\right)^\frac{2_{\alpha}^*(s)}{2_{\alpha}^*(s)-2} = \frac{\alpha-s}{2(n-s)} S(n,\alpha,\gamma,s)^{\frac{n-s}{\alpha-s}}. $$
Again,  if equality holds, then 
$0<c_w(\Psi) \le \sup\limits_{t\ge0}\Psi(tw) = \sup\limits_{t\ge0}\bar{f}(t)$, and if $t_1,t_2>0$ are points where the respective suprema are attained, then a  contradiction is reached since 
$$\bar{f}(t_1) -\frac{ t_1^{ 2_{\alpha}^*}  }{2_{\alpha}^*}\int_{\R^n} |u|^{2_{\alpha}^*}  dx = \bar{f}(t_2).$$
Therefore,
\begin{equation*} 
0 < c_w(\Psi) < c^\star = \text{min} \left\{ \frac{\alpha}{2n} S(n,\alpha,\gamma,0)^{\frac{n}{\alpha}}, \frac{\alpha-s}{2(n-s)} S(n,\alpha,\gamma,s)^{\frac{n-s}{\alpha-s}}   \right\}.
\end{equation*}
Finally, the existence of a Palais-Smale sequence at that level follows immediately from Lemma \ref{Theorem MPT- Ambrosetti-Rabinowitz version}
\end{proof}

\subsection{Analysis of the Palais-Smale sequences}

We now study the concentration properties of weakly null Palais-Smale sequences. For $\delta>0$, we shall write $B_\delta := \left\{ x \in \R^n : |x| < \delta \right\}.$
\begin{proposition} \label{Proposition lim of sobolev term in small ball when miniming sequence converges weakly to zero}
Let  $0 \le \gamma < \gamma_H \text{ and } 0 < s<\alpha.$ Assume that $(w_k)_{k \in \mathbb{N}}$ is a Palais-Smale sequence  of $\Psi$ at energy level  $c \in (0, c^\star )$. 
If $w_k \rightharpoonup 0 $ in $X^{\alpha} (\R_+^{n+1}) \text{ as } {k \to \infty}$, then there exists a positive constant $\epsilon_0= \epsilon_0 (n, \alpha, \gamma,c,s)>0$ such that for every $\delta>0$, one of the following holds:
\begin{enumerate}
\item $\limsup\limits_{k \to \infty} \int_{B_{\delta}}|u_k|^{2^*_{\alpha}}dx = \limsup\limits_{k \to \infty} \int_{B_{\delta}} \frac{|u_k|^{2^*_{\alpha}(s)}}{|x|^s}dx =0;$
\item $\limsup\limits_{k \to \infty} \int_{B_{\delta}} |u_k|^{2^*_{\alpha}}dx\,\, \hbox{and} \,\, \limsup\limits_{k \to \infty} \int_{B_{\delta}} \frac{|u_k|^{2^*_{\alpha}(s)}}{|x|^s}dx \ge \epsilon_0,$
\end{enumerate}

\end{proposition}

The proof of Proposition \ref{Proposition lim of sobolev term in small ball when miniming sequence converges weakly to zero} requires the following two lemmas.

\begin{lemma} \label{Lemma limi of Hardy-Sobolev and grad terms are zero on D}
Let $(w_k)_{k \in \mathbb{N}}$ be a Palais-Smale sequence as in Proposition \ref{Proposition lim of sobolev term in small ball when miniming sequence converges weakly to zero}. 
If $w_k \rightharpoonup 0 $ in $X^{\alpha} (\R_+^{n+1}),$  then for any  $D  \subset\subset \R^n \setminus \{0\},$ there exists a subsequence of $(w_k)_{k \in \mathbb{N}}$, still denoted by $(w_k)_{k \in \mathbb{N}}$, such that

\begin{equation}\label{lim of Hardy and Hardy-Sobolev terms on D* are zero}
\lim\limits_{k \to \infty}  \int_{D} \frac{|u_k|^2}{|x|^\alpha} dx =\lim\limits_{k \to \infty} \int_{D} \frac{|u_k|^{2_{\alpha}^*(s)}}{|x|^{s}}  dx =0
\end{equation}  
  and
   \begin{equation}\label{lim of Sobolev and grad terms on D are zero}
\lim\limits_{k \to \infty}  \int_{D} |u_k|^{2_{\alpha}^*} dx=  \lim\limits_{k \to \infty} \int_D |(-\Delta)^{\frac{\alpha}{4}} u_k|^2dx = 0 ,
\end{equation}
where $u_k:=w_k(.,0)$ for all $k\in \mathbb{N}.$  
  \end{lemma}
\begin{proof}[Proof of Lemma \ref{Lemma limi of Hardy-Sobolev and grad terms are zero on D}]
Fix  $D \subset\subset \R^n \setminus \{0\},$ and note that 
the following fractional Sobolev embedding is compact:
$$H^{\frac{\alpha}{2}}(\R^n)\hookrightarrow L^q(D)  \text{ for every } 1\le q < 2^*_\alpha .$$  

Using the trace inequality (\ref{trace inequality between extension norm and fractional sobolev}), and the assumption that $w_k \rightharpoonup 0 $ in $X^{\alpha} (\R_+^{n+1}),$  we get that 
$$ {u_k \to 0} \quad \text{ strongly for every } 1\le q < 2^*_\alpha.$$
On the other hand, the fact that $|x|^{-1}$ is bounded on $D \subset\subset \R^n \setminus \{0\}$ implies that there exist constants $C_1,C_2 >0$ such that
$$0 \le \lim\limits_{k \to \infty}  \int_{D} \frac{|u_k|^2}{|x|^\alpha} dx \le C_1 \lim\limits_{k \to \infty}  \int_{D} |u_k|^2 dx   $$
and 
$$0 \le \lim\limits_{k \to \infty} \int_{D} \frac{|u_k|^{2_{\alpha}^*(s)}}{|x|^{s}}  dx  \le C_2 \lim\limits_{k \to \infty} \int_{D} |u_k|^{2_{\alpha}^*(s)}  dx  . $$
Since $s \in (0,\alpha),$ we have that $1\le 2, 2^*_\alpha (s)< 2^*_\alpha.$ Thus, (\ref{lim of Hardy and Hardy-Sobolev terms on D* are zero}) holds.

To show 
 (\ref{lim of Sobolev and grad terms on D are zero}), we  
let $\eta \in C_0^{\infty}(\R_+^{n+1})$ be a cut-off function such that  $\eta_*:=\eta(.,0) \in C_0^{\infty}(\R^n \setminus \{0\}),$  $\eta_* \equiv 1$  in $D$ and $0 \le \eta \le 1$ in $\R_+^{n+1}.$ 
 We first note that  
 \begin{equation} \label{estimate for grad term with cut-off (eta w_k)}
 k_{\alpha}  \int_{\R_+^{n+1}} y^{1-{\alpha}} |\nabla ( \eta w_k )|^2 dxdy= k_{\alpha} \int_{\R_+^{n+1}}  y^{1-{\alpha}} |  \eta \nabla w_k|^2 dxdy+ o(1).
\end{equation}
Indeed, apply the following elementary inequality for vectors $X,Y$ in $\R^{n+1}$, 
$$\left | |X+Y|^2 - |X|^2 \right | \le C (|X||Y|+|Y|^2),$$
with $ X=  y^{\frac{1-\alpha}{2}}w_k  \nabla\eta $ and $Y=  y^{\frac{1-\alpha}{2}} \eta \nabla w_k $,  to get for all $k \in \mathbb{N}$, that
\begin{align*}
 \left|  y^{1-\alpha} | \nabla( \eta w_k )|^2 -   y^{1-\alpha} |  \eta \nabla w_k|^2 \right| \le C \left(  y^{1-\alpha} |w_k  \nabla\eta|  | \eta \nabla w_k |+  y^{1-\alpha} |\eta \nabla w_k |^2 \right). 
\end{align*}
By H\"older's  inequality, we get  
\begin{equation} \label{Holder for gradient 1}
\begin{aligned}
&\left | \int_{\R_+^{n+1}}  y^{1-\alpha} | \nabla( \eta w_k )|^2 dxdy-  \int_{\R_+^{n+1}}  y^{1-\alpha} |  \eta \nabla w_k|^2  dxdy \right| \\
& \le \ C_3 \|w_k\|_{X^{\alpha} (\R_+^{n+1})}  \ (\int_{\text{Supp} (\nabla \eta) }  y^{1-\alpha} |w_k|^2 dxdy  )^\frac{1}{2} + C_3 \int_{\text{Supp} (\nabla \eta) }  y^{1-\alpha} |w_k|^2 dxdy \\
& \le C_4 \left[(\int_{\text{Supp} (\nabla \eta) }  y^{1-\alpha} |w_k|^2 dxdy)^{\frac{1}{2}} +\int_{\text{Supp} (\nabla \eta) }  y^{1-\alpha} |w_k|^2 dxdy  \right].
\end{aligned}
\end{equation}
Since the embedding $H^1(\text{Supp} (\nabla \eta), y^{1-\alpha}) \hookrightarrow L^2(\text{Supp} (\nabla \eta), y^{1-\alpha})$  is compact,
and  $w_k \rightharpoonup 0 $ in $X^{\alpha} (\R_+^{n+1}),$ we get that  $$\int_{\text{Supp} (\nabla \eta) }  y^{1-\alpha} |w_k|^2 dxdy = o(1),$$
which gives  

$$\int_{\R_+^{n+1}}  y^{1-\alpha} | \nabla( \eta w_k )|^2 dxdy= \int_{\R_+^{n+1}}  y^{1-\alpha} |  \eta \nabla w_k|^2  dxdy+o(1). $$
Thus, (\ref{estimate for grad term with cut-off (eta w_k)}) holds.

Now recall that the sequence $(w_k)_{k \in \mathbb{N}}$ has the following property:
  \begin{equation} \label{Derivative of the functional converges strongly to zero}
 \lim\limits_{k \to \infty}  \Psi'(w_k)=0 \text{ strongly in } (X^{\alpha} (\R_+^{n+1}))' .
 \end{equation}
Since $\eta^2 w_k \in X^{\alpha} (\R_+^{n+1})$ for all $k \in \mathbb{N}$, we can use it as 
a test function in (\ref{Derivative of the functional converges strongly to zero}) to get that 
  \begin{equation*}
 \begin{aligned} 
 o(1) &= \langle \Psi' (w_k), \eta^2 w_k\rangle \\
 & = k_{\alpha} \int_{\R_+^{n+1}} y^{1-\alpha} \langle \nabla w_k , \nabla (\eta^2 w_k )\rangle dxdy - \gamma \int_{\R^n} \frac{\eta_*^2 |u_k|^2  }{|x|^{\alpha}}dx - \int_ {\R^n} \eta_*^2 |u_k|^{2_{\alpha}^*}   dx - \int_ {\R^n} \frac{\eta_*^2 |u_k|^{2_{\alpha}^*(s)}   }{|x|^s}dx.
 \end{aligned}
 \end{equation*}
 Regarding the first term, we have 
$$k_{\alpha} \int_{\R_+^{n+1}} y^{1-\alpha} \langle \nabla w_k , \nabla (\eta^2 w_k )\rangle dxdy =  k_\alpha \int_{\R_+^{n+1}} y^{1-\alpha} |\eta \nabla w_k |^2 dxdy  +k_{\alpha} \int_{\R_+^{n+1}} y^{1-\alpha} w_k \langle \nabla (\eta^2) ,\nabla w_k \rangle   dxdy .$$
From H\"older's inequality, and the fact that ${w_k \to 0}$ in  $L^2(\text{Supp} (|\nabla \eta|), y^{1-\alpha}),$ it follows that as $k \to \infty$, 

\begin{align*}
& \left| k_{\alpha} \int_{\R_+^{n+1}} y^{1-\alpha} \langle \nabla w_k , \nabla (\eta^2 w_k )\rangle dxdy -  k_\alpha \int_{\R_+^{n+1}} y^{1-\alpha} |\eta \nabla w_k  |^2 dxdy\right| = \left| k_{\alpha} \int_{\R_+^{n+1}} y^{1-\alpha} w_k \langle \nabla( \eta^2) ,\nabla w_k \rangle   dxdy \right| \\
&\qquad \qquad \qquad \le k_{\alpha} \int_{\R_+^{n+1}} y^{1-\alpha} |w_k|  |\nabla (\eta^2)| |\nabla w_k|   dxdy \le C  \int_{\text{Supp}(|\nabla \eta|)} y^{1-\alpha} |w_k|   |\nabla w_k| dxdy \\
& \qquad \qquad \qquad  \le C \| w_k \|_{X^{\alpha}(\R^{n+1}_+)} \left(\int_{\text{Supp}(|\nabla \eta|)} y^{1-\alpha} |w_k|^2  dxdy\right)^{\frac{1}{2}} \\
&\qquad \qquad \qquad  = o(1). 
\end{align*}
Thus, we have proved that  
$$k_{\alpha} \int_{\R_+^{n+1}} y^{1-\alpha} \langle \nabla w_k , \nabla (\eta^2 w_k )\rangle dxdy =  k_\alpha \int_{\R_+^{n+1}} y^{1-\alpha} |\eta \nabla  w_k  |^2 dxdy  + o(1).$$

Using the above estimate coupled with (\ref{estimate for grad term with cut-off (eta w_k)}), we  obtain
\begin{equation}
 \begin{aligned} 
 o(1) &= \langle \Psi' (w_k), \eta^2 w_k\rangle \\
 & = k_\alpha \int_{\R_+^{n+1}} y^{1-\alpha} |\eta \nabla w_k  |^2 dxdy  - \gamma \int_{K} \frac{\eta_*^2 |u_k|^2  }{|x|^{\alpha}}dx - \int_ {\R^n} \eta_*^2 |u_k|^{2_{\alpha}^*}   dx - \int_ {K} \frac{\eta_*^2 |u_k|^{2_{\alpha}^*(s)}   }{|x|^s}dx +o(1)\\
 &=k_\alpha \int_{\R_+^{n+1}} y^{1-\alpha} |\nabla (\eta w_k ) |^2 dxdy   - \int_ {\R^n} \eta_*^2 |u_k|^{2_{\alpha}^*}   dx +o(1)\\
 & \ge \|\eta w_k\|^2 - \int_ {\R^n} \eta_*^2 |u_k|^{2_{\alpha}^*}   dx +o(1), \quad \text{ as } {k \to \infty},
 \end{aligned}
 \end{equation}
 where $K= \text{Supp}(\eta_*).$ Therefore,
\begin{equation}
\|\eta w_k\|^2 \le  \int_ {\R^n} |\eta_*u_k|^2 |u_k|^{2_{\alpha}^*-2}   dx +o(1) \quad \text{ as } {k \to \infty}.
\end{equation}
By H\"older's inequality, and using  the definition of $S(n,\alpha,\gamma,0),$ we then get that
\begin{equation}
\begin{aligned}
\|\eta w_k\|^2  &\le  \left(\int_ {\R^n} |\eta_*u_k|^{2_{\alpha}^*}dx \right)^ {\frac{2}{2_{\alpha}^*}} \left(\int_ {\R^n} |u_k|^{2_{\alpha}^*} dx \right)^{\frac{2_{\alpha}^*-2}{2_\alpha^*}} +o(1)\\
& \le  S(n,\alpha,\gamma,0)^{-1} \| \eta w_k \|^2 \left(\int_ {\R^n} |u_k|^{2_{\alpha}^*} dx \right)^{\frac{2_{\alpha}^*-2}{2_\alpha^*}} +o(1). 
\end{aligned}
\end{equation}
Thus,
\begin{equation}\label{estimate for new norm with  (eta w_k) by o(1) }
\left[ 1- S(n,\alpha,\gamma,0)^{-1} \left(\int_ {\R^n} |u_k|^{2_{\alpha}^*} dx \right)^{\frac{2_{\alpha}^*-2}{2_\alpha^*}} \right] \| \eta w_k \|^2 \le o(1).
\end{equation}
In addition, it follows from (\ref{P-S condition (Lim) on minimizing sequence}) that 
$$\Psi(w_k) - \frac{1}{2} \langle \Psi'(w_k), w_k\rangle = c+ o(1),$$
that is, 
\begin{equation} \label{Summation of Hardy and Hardy-Sobolev terms for P-S sequence is c+o(1) }
(\frac{1}{2} - \frac{1}{2_{\alpha}^*})  \int_{\R^n} |u_k|^{2_{\alpha}^*} dx +  (\frac{1}{2} - \frac{1}{2_{\alpha}^*(s)})   \int_{\R^n} \frac{|u_k|^{2_{\alpha}^*(s)}}{|x|^{s}}dx  = c+ o(1), 
\end{equation}
from which follows that 
  \begin{equation} \label{upper bounded for Hardy term P-S sequence 2n/c}
\int_{\R^n} |u_k|^{2_{\alpha}^*} dx \le \frac{2n}{\alpha} c +o(1), \quad \text{ as } {k \to \infty}.
 \end{equation}
 Plugging (\ref{upper bounded for Hardy term P-S sequence 2n/c}) into (\ref{estimate for new norm with  (eta w_k) by o(1) }), we obtain that 
  $$ \left[ 1- S(n,\alpha,\gamma,0)^{-1} (\frac{2n}{\alpha} c )^{\frac{\alpha}{n}} \right] \| \eta w_k \|^2 \le o(1), \text{ as } {k \to \infty}.  $$
 
On the other hand, by the upper bound (\ref{Definition of C^star}) on $c,$ we have that 
$$ c <  \frac{\alpha}{2n} S(n,\alpha,\gamma,0)^{\frac{n}{\alpha}}.$$ 
This yield that $  1- S(n,\alpha,\gamma,0)^{-1} (\frac{2n}{\alpha} c )^{\frac{\alpha}{n}} > 0,$ and therefore, $ \lim\limits_{k \to \infty} \| \eta w_k \|^2 = 0. $

Using  (\ref{extension norm}) and  (\ref{comparable norms}), we obtain that 
$$\lim\limits_{k \to \infty} \int_{\R^n} |(-\Delta)^{\frac{\alpha}{4}} (\eta_* u_k)|^2dx = \lim\limits_{k \to \infty}  k_{\alpha}\int_{\R^{n+1}_+} y^{1-\alpha} |\nabla (\eta w_k)|^2  dxdy =0.$$
It also follows from the definition of $S(n, \alpha, \gamma,0) $  that 
$\lim\limits_{k \to \infty}  \int_{\R^n} | \eta_* u_k |^{2^*_{\alpha}}dx =0$, hence,
$$\lim\limits_{k \to \infty} \int_{\R^n} |(-\Delta)^{\frac{\alpha}{4}} (\eta_* u_k)|^2dx= \lim\limits_{k \to \infty}  \int_{\R^n} | \eta_* u_k |^{2^*_{\alpha}} dx=0.$$
Since 
${\eta_*}_{|_D} \equiv 1,$ the last equality yields  (\ref{lim of Sobolev and grad terms on D are zero}).  
\end{proof}

\begin{lemma} \label{Lemma relation between theta zeta and mu with S(n,alpha,gamma,s)}
Let $(w_k)_{k \in \mathbb{N}}$ be Palais-Smale sequence as in Proposition \ref{Proposition lim of sobolev term in small ball when miniming sequence converges weakly to zero} and let  $u_k:=Tr (w_k)= w_k(.,0)$. For any $\delta > 0,$ set 
 \begin{equation} \label{Definition theta, zeta and mu}
 \begin{aligned}
 &\theta:= \limsup\limits_{k \to \infty} \int_{B_{\delta}} |u_k|^{2^*_{\alpha}}dx; \qquad  \zeta:= \limsup\limits_{k \to \infty} \int_{B_{\delta}} \frac{|u_k|^{2^*_{\alpha}(s)}}{|x|^s}dx \, \,  {\rm and} \\
&\mu:= \limsup\limits_{k \to \infty} \int_{B_\delta} \left( |(-\Delta)^{\frac{\alpha}{4}}  u_k|^2dx - \gamma  \frac{|u_k|^2}{|x|^{\alpha}} \right) dx ,
 \end{aligned}
 \end{equation}
 where $u:= w(.,0)$. If $w_k \rightharpoonup 0 $ in $X^{\alpha} (\R_+^{n+1})$ as ${k \to \infty},$ then the following hold:
\begin{enumerate}
\item $ \theta^{\frac{2}{2^*_{\alpha}}} \le S(n,\alpha,\gamma,0)^{-1} \mu \quad \text{ and } \quad  \zeta^{\frac{2}{2^*_{\alpha}(s)}} \le S(n,\alpha,\gamma,s)^{-1} \mu.$
\item $\mu \le \theta + \zeta.$
\end{enumerate}
\end{lemma}
\begin{proof}[Proof of Lemma \ref{Lemma relation between theta zeta and mu with S(n,alpha,gamma,s)}]
First note that it follows from Lemma \ref{Lemma limi of Hardy-Sobolev and grad terms are zero on D} that $\theta, \zeta$ and $\mu $ are well-defined and are independent of the choice of $\delta > 0.$
Let now $\eta \in C_0^{\infty}(\R_+^{n+1})$ be a cut-off function such that 
$\eta_*:=\eta(.,0) \equiv 1$  in $B_\delta,$ and $0 \le \eta \le 1$ in $\R_+^{n+1}.$

1.  Since $\eta w_k \in X^{\alpha}(\R_+^{n+1})$, we get from the definition of $S(n,\alpha,\gamma,s)$ that
 \begin{equation}\label{fractional Hardy-Sobolev inequality when s=0}
 S(n,\alpha,\gamma,0)(\int_{\R^n} |\eta_* u_k|^{2_{\alpha}^*}dx)^\frac{2}{2_{\alpha}^*}\le    k_{\alpha}\int_{\R^{n+1}_+} y^{1-\alpha} |\nabla(\eta w_k)|^2  dxdy - \gamma \int_{\R^n} \frac{|\eta_* u_k|^2}{|x|^{\alpha}} dx.
 \end{equation}

On the other hand, from the definition of  $\eta$ and (\ref{extension norm}), it follows that 
\begin{align*}
&k_{\alpha}\int_{\R^{n+1}_+} y^{1-\alpha} |\nabla(\eta w_k)|^2  dxdy - \gamma \int_{\R^n} \frac{|\eta_* u_k|^2}{|x|^{\alpha}} dx = \int_{\R^n} \left( |(-\Delta)^{\frac{\alpha}{4}}  (\eta_* u_k)|^2 - \gamma  \frac{|\eta_*u_k|^2}{|x|^{\alpha}} \right) dx \\
&=   \int_{B_\delta} \left( |(-\Delta)^{\frac{\alpha}{4}}  u_k|^2 - \gamma  \frac{|u_k|^2}{|x|^{\alpha}} \right) dx + \int_{\text{ Supp}(\eta_*) \setminus B_\delta} \left( |(-\Delta)^{\frac{\alpha}{4}}  (\eta_* u_k)|^2 - \gamma  \frac{|\eta_*u_k|^2}{|x|^{\alpha}} \right)   dx,
\end{align*}
and$$(\int_{B_\delta} |u_k|^{2_{\alpha}^*}dx)^\frac{2}{2_{\alpha}^*} \le  (\int_{\R^n} |\eta_* u_k|^{2_{\alpha}^*}dx)^\frac{2}{2_{\alpha}^*}. $$

Note that  $ \text{ Supp}(\eta_*) \setminus B_\delta \subset\subset \R^n \setminus \{0\}.$ Therefore, taking the upper limits at both sides of (\ref{fractional Hardy-Sobolev inequality when s=0}), and using Lemma \ref{Lemma limi of Hardy-Sobolev and grad terms are zero on D}, we get that
 
$$S(n,\alpha,\gamma,0) (\int_{B_\delta} |u_k|^{2_{\alpha}^*}dx)^\frac{2}{2_{\alpha}^*}\le   \int_{B_\delta} \left( |(-\Delta)^{\frac{\alpha}{4}}  u_k|^2dx - \gamma  \frac{|u_k|^2}{|x|^{\alpha}} \right) dx  +o(1) \ \text{ as }{k \to \infty},$$

which gives 

$$\theta^{\frac{2}{2^*_{\alpha}}} \le S(n,\alpha,\gamma,0)^{-1} \mu.$$

Similarly, we can prove that $$\zeta^{\frac{2}{2^*_{\alpha}(s)}} \le S(n,\alpha,\gamma,s)^{-1} \mu.$$

2.  Since $\eta^2 w_k \in X^{\alpha}(\R_+^{n+1})$ and  $\langle \Psi'(w_k) , \eta^2 w_k \rangle =o(1) \text{ as } {k \to \infty}$, we have 

\begin{equation}\label{Estimating Psi with eta w_k}
\begin{aligned}
o(1) &= \langle \Psi'(w_k) , \eta^2 w_k \rangle  \\
&=k_{\alpha} \int_{\R_+^{n+1}} y^{1-\alpha} \langle \nabla w_k , \nabla (\eta^2 w_k ) \rangle dxdy - \gamma \int_{\R^n} \frac{|\eta_*u_k|^2}{|x|^{\alpha}} dx - \int_ {\R^n} \eta^2_*|u_k|^{2_{\alpha}^*}   dx - \int_{\R^n} \frac{\eta^2_* |u_k|^{2_{\alpha}^*(s)}}{|x|^{s}}dx \\ 
& = \left(k_{\alpha} \int_{\R_+^{n+1}} y^{1-\alpha}  |\eta \nabla w_k|^2   dxdy - \gamma \int_{\R^n} \frac{|\eta_*u_k|^2}{|x|^{\alpha}} dx \right) - \int_ {\R^n} \eta^2_* |u_k|^{2_{\alpha}^*}   dx - \int_{\R^n} \frac{\eta^2_* |u_k|^{2_{\alpha}^*(s)}}{|x|^{s}}dx \\
& \quad + k_{\alpha} \int_{\R_+^{n+1}} y^{1-\alpha} w_k \langle \nabla (\eta^2) ,\nabla w_k \rangle   dxdy.
\end{aligned}
\end{equation}

By H\"older's inequality, and the fact that ${w_k \to 0}$ in  $L^2(\text{Supp} (|\nabla \eta|), y^{1-\alpha}),$ we obtain that 

\begin{align*}
 \left| k_{\alpha} \int_{\R_+^{n+1}} y^{1-\alpha} w_k \langle \nabla (\eta^2) ,\nabla w_k \rangle   dxdy \right| &\le k_{\alpha} \int_{\R_+^{n+1}} y^{1-\alpha} |w_k|  |\nabla (\eta^2)| |\nabla w_k|   dxdy \\ 
&\le C  \int_{\text{Supp}(|\nabla \eta|)} y^{1-\alpha} |w_k|   |\nabla w_k| dxdy \\
& \le C \| w_k \|_{X^{\alpha}(\R^{n+1}_+)} \|w_k\|_{L^2(\text{Supp} (|\nabla \eta|), y^{1-\alpha})}\\
& \le o(1) \quad \hbox{as ${k \to \infty}.$}
\end{align*}
Plugging the above estimate into (\ref{Estimating Psi with eta w_k}) and using (\ref{extension norm}), we get that  

\begin{equation*}
\begin{aligned}
o(1) &= \langle \Psi'(w_k) , \eta^2 w_k \rangle  \\
& = \left(k_{\alpha} \int_{\R_+^{n+1}} y^{1-\alpha}  |\nabla (\eta w_k)|^2   dxdy - \gamma \int_{\R^n} \frac{|\eta_*u_k|^2}{|x|^{\alpha}} dx \right) - \int_ {\R^n} \eta^2_* |u_k|^{2_{\alpha}^*}   dx - \int_{\R^n} \frac{\eta^2_* |u_k|^{2_{\alpha}^*(s)}}{|x|^{s}}dx \\
& = \int_{\R^n} \left( |(-\Delta)^{\frac{\alpha}{4}} (\eta_* u_k)|^2  - \gamma \frac{|\eta_*u_k|^2}{|x|^{\alpha}} \right) dx  - \int_ {\R^n} \eta^2_* |u_k|^{2_{\alpha}^*}   dx - \int_{\R^n} \frac{\eta^2_* |u_k|^{2_{\alpha}^*(s)}}{|x|^{s}}dx \\
& \ge \int_{B_\delta} \left( |(-\Delta)^{\frac{\alpha}{4}}  u_k|^2  - \gamma \frac{|u_k|^2}{|x|^{\alpha}} \right) dx  - \int_ {B_\delta}  |u_k|^{2_{\alpha}^*}   dx - \int_{B_\delta} \frac{ |u_k|^{2_{\alpha}^*(s)}}{|x|^{s}}dx \\
&- \int_{\text{ Supp}(\eta_*) \setminus B_\delta} \left(\gamma  \frac{|\eta_*u_k|^2}{|x|^{\alpha}} dx  +  \eta^2_* |u_k|^{2_{\alpha}^*}   dx + \frac{\eta^2_* |u_k|^{2_{\alpha}^*(s)}}{|x|^{s}} \right) dx +o(1).
\end{aligned}
\end{equation*}

Noting that $\text{ Supp}(\eta_*) \setminus B_\delta \subset\subset \R^n \setminus \{0\},$ and taking the upper limits on both sides, we get that $\mu \le \theta+\zeta.$
\end{proof}

\begin{proof}[Proof of Proposition \ref{Proposition lim of sobolev term in small ball when miniming sequence converges weakly to zero}]

It follows from Lemma \ref{Lemma relation between theta zeta and mu with S(n,alpha,gamma,s)} that 

$$ \theta^{\frac{2}{2^*_\alpha}} \le S(n,\alpha,\gamma,0)^{-1} \mu \le  S(n,\alpha,\gamma,0)^{-1} \theta + S(n,\alpha,\gamma,0)^{-1} \zeta, $$
which gives

\begin{equation} \label{estimate for theta and zeta , using estimates in Lemmas}
\begin{aligned}
 &\theta^{\frac{2}{2^*_\alpha}} (1- S(n,\alpha,\gamma,0)^{-1}    \theta^{\frac{2^*_\alpha - 2}{2^*_\alpha}}) \le S(n,\alpha,\gamma,0)^{-1} \zeta.
 \end{aligned}
\end{equation}
On the other hand, by (\ref{Summation of Hardy and Hardy-Sobolev terms for P-S sequence is c+o(1) }), we have

$$\theta \le \frac{2n}{\alpha} c. $$

Substituting the last inequality into (\ref{estimate for theta and zeta , using estimates in Lemmas}), we get that 

$$ (1- S(n,\alpha,\gamma,0)^{-1}  (\frac{2n}{\alpha} c)^{\frac{\alpha }{n}})  \theta^{\frac{2}{2^*_\alpha}}  \le S(n,\alpha,\gamma,0)^{-1} \zeta. $$ 
Recall that the upper bounded (\ref{Definition of C^star}) on $c$ implies that

$$1- S(n,\alpha,\gamma,0)^{-1}  (\frac{2n}{\alpha} c)^{\frac{\alpha }{n}} > 0.$$
Therefore, there exists $\delta_1= \delta_1(n,\alpha,\gamma,c)>0$ such that $ \theta^{\frac{2}{2^*_\alpha}}  \le \delta_1 \zeta.$ Similarly, 
there exists $\delta_2= \delta_2(n,\alpha,\gamma,c,s)>0$ such that $  \zeta^{\frac{2}{2^*_\alpha(s)}}  \le \delta_2 \theta. $ These two inequalities yield that there exists $\epsilon_0= \epsilon_0 (n, \alpha, \gamma,c,s)>0$ such that 
\begin{equation} \label{Definition of epsilon_0}
 \text{ either } \quad  \theta= \zeta = 0 \quad  \text{ or } \ \quad \{\theta \ge \epsilon_0 \text{ and } \zeta \ge \epsilon_0 \}.
 \end{equation}
It follows from the definition of $\theta$ and $\zeta$ that 

 $$\text{ either }\limsup\limits_{k \to \infty} \int_{B_{\delta}} |u_k|^{2^*_{\alpha}}dx = \limsup\limits_{k \to \infty} \int_{B_{\delta}} \frac{|u_k|^{2^*_{\alpha}(s)}}{|x|^s}dx =0;$$
$$ \text{ or } \quad \ \ \limsup\limits_{k \to \infty} \int_{B_{\delta}} |u_k|^{2^*_{\alpha}}dx \ge \epsilon_0 \quad {\rm and} \quad \limsup\limits_{k \to \infty} \int_{B_{\delta}} \frac{|u_k|^{2^*_{\alpha}(s)}}{|x|^s}dx \ge \epsilon_0.$$
\end{proof} 
\subsection{End of proof of Theorem \ref{Theorem Main result in extended form}}

We shall first eliminate the possibility of a zero weak limit for the Palais-Smale sequence of $\Psi$, then we prove that the nontrivial weak limit is indeed a weak solution of Problem (\ref{Main problem.prime}). In the sequel $(w_k)_{k \in \mathbb{N}}$ will denote the Palais-Smale sequence for $\Psi$ obtained in Proposition \ref{Proposition lim of sobolev term in small ball when miniming sequence converges weakly to zero}. 

First we show that 
\begin{equation}\label{limsup}
 \limsup\limits_{k \to \infty} \int_{\R^n} |u_k|^{2^*_\alpha} dx > 0.
 \end{equation} 
Indeed, otherwise   
$\lim\limits_{k \to \infty} \int_{\R^n} |u_k|^{2^*_\alpha} dx =0,$ which once combined with the fact 
that $\langle \Psi'(w_k),w_k \rangle \to 0$ yields that 
$ \|w_k\|^2 = \int_ {\R^n} \frac{ |u_k|^{2_{\alpha}^*(s)}}{|x|^s} dx +o(1). $

By combining this estimate with the definition of $S(n, \alpha, \gamma, s)$, we obtain
$$
\left(\int_ {\R^n} \frac{ |u_k|^{2_{\alpha}^*(s)}}{|x|^s}dx\right)^{\frac{2}{2_\alpha^*(s)}}\le S(n, \alpha, \gamma, s)^{-1} \|w_k\|^2 \le S(n, \alpha, \gamma, s)^{-1} \int_ {\R^n} \frac{ |u_k|^{2_{\alpha}^*(s)}}{|x|^s} dx +o(1), $$
which implies that
$$\left(\int_ {\R^n} \frac{ |u_k|^{2_{\alpha}^*(s)}}{|x|^s}dx\right)^{\frac{2}{2_\alpha^*(s)}} \left[ 1-     S(n, \alpha, \gamma, s)^{-1}  (\int_ {\R^n} \frac{ |u_k|^{2_{\alpha}^*(s)}}{|x|^s} dx )^{\frac{2^*_\alpha(s)-2}{2^*_\alpha(s)}} \right] \le o(1).$$
It follows from  (\ref{Definition of C^star}) and (\ref{Summation of Hardy and Hardy-Sobolev terms for P-S sequence is c+o(1) })  that as ${k \to \infty}$, 
$$\int_ {\R^n} \frac{ |u_k|^{2_{\alpha}^*(s)}}{|x|^s} dx= 2c \frac{n-s}{\alpha-s} +o(1) \quad  \text{ and } \quad (1- S(n,\alpha,\gamma,s)^{-1}  (2 c \frac{n-s}{\alpha-s})^{\frac{\alpha-s}{n-s}}) >0.$$
Hence,
\begin{equation} \label{lim of Hardy-Sobolev term is zero (Contradiction)}
\lim\limits_{k \to \infty}\int_ {\R^n} \frac{ |u_k|^{2_{\alpha}^*(s)}}{|x|^s} dx = 0.
\end{equation}

Using that $\lim\limits_{k \to \infty} \int_{\R^n} |u_k|^{2^*_\alpha} dx =0,$ 
in conjunction with 
(\ref{lim of Hardy-Sobolev term is zero (Contradiction)}) and  
(\ref{Summation of Hardy and Hardy-Sobolev terms for P-S sequence is c+o(1) }), we get that  $c+o(1) = 0,$
which contradicts the fact that $c>0.$ This completes the proof of (\ref{limsup}).

Now, we show that for small enough $\epsilon >0$,  there exists another Palais-Smale sequence $(v_k)_{k \in \mathbb{N}}$ for $\Psi$ satisfying the properties of Proposition \ref{Proposition lim of sobolev term in small ball when miniming sequence converges weakly to zero}, which is also bounded in $X^\alpha(\R^{n+1}_+)$ and satisfies
\begin{equation}\label{epsilon}
\int_{B_1} |v_k(x,0)|^{2^*_\alpha} dx =\epsilon \quad \hbox{for all $k \in \mathbb{N}.$} 
\end{equation}

For that, consider $\epsilon_0$ as given in Proposition \ref{Proposition lim of sobolev term in small ball when miniming sequence converges weakly to zero}. Let  $\beta = \limsup\limits_{k \to \infty} \int_{\R^n} |u_k|^{2^*_\alpha} dx$, which is positive by (\ref{limsup}).  
Set $\epsilon_1 := \text{min} \{\beta , \frac{\epsilon_0}{2}\}$
and fix $\epsilon \in (0,\epsilon_1).$ Up to a subsequence, there exists by continuity a sequence of radii $(r_k )_k$  such that $ \int_{B_{r_k}} |u_k|^{2^*_\alpha} dx =\epsilon$
for each $k \in \mathbb{N}.$ Let now $$ v_k(x,y) := r_k^{\frac{n-\alpha}{2}} w_k(r_k x, r_k y) \quad  \text{ for } x \in \R^n \text{ and } y \in \R_+ .$$
It is clear that
\begin{equation}\label{epsilon}
\int_{B_1} |v_k(x,0)|^{2^*_\alpha} dx= \int_{B_{r_k}} |u_k|^{2^*_\alpha} dx =\epsilon \quad \hbox{for all $k \in \mathbb{N}.$} 
\end{equation}
It is easy to check that $(v_k)_{k \in \mathbb{N}}$ is also a Palais-Smale sequence for $\Psi$ that satisfies the properties of Proposition \ref{Proposition lim of sobolev term in small ball when miniming sequence converges weakly to zero}. We now show that it is bounded in $X^\alpha(\R^{n+1}_+)$ .

Since  $ (v_k)_{k \in \mathbb{N}} $ is a Palais-Smale sequence, there exists positive constants $ C_1, C_2 >0$ such that 
\begin{equation}\label{Use P-S condition (v_k) to prove boundedness}
\begin{aligned} 
C_1 + C_2 \|v_k\| &\ge \Psi(v_k) - \frac{1}{2^*_\alpha(s)} \langle \Psi'(v_k), v_k \rangle \\
&  \ge\left( \frac{1}{2} - \frac{1}{2^*_\alpha(s)} \right) \|v_k\|^2+ \left(\frac{1}{2^*_\alpha}  - \frac{1}{2^*_\alpha(s)} \right) \int_{\R^n} |v_k(x,0)|^{2^*_\alpha} dx\\
& \ge  \left( \frac{1}{2} - \frac{1}{2^*_\alpha(s)} \right) \|v_k\|^2.
\end{aligned}
\end{equation}
The last inequality holds since $2<2^*_\alpha(s) < 2^*_\alpha.$ Combining (\ref{Use P-S condition (v_k) to prove boundedness}) with (\ref{comparable norms}), we obtain that $ (v_k)_{k \in \mathbb{N}} $ is bounded in  $X^\alpha(\R^{n+1}_+).$ 

It follows that  there exists a subsequence -- still denoted by $v_k$ -- such that $ v_k \rightharpoonup v \text{ in } X^\alpha(\R^{n+1}_+)$ as ${k \to \infty}.$ We claim that $v$ is a  nontrivial weak solution of (\ref{Main problem.prime}). Indeed, 
if $v \equiv 0$, then  Proposition \ref{Proposition lim of sobolev term in small ball when miniming sequence converges weakly to zero} yields that 
 $$ \text{ either } \ \limsup\limits_{k \to \infty} \int_{B_1} |v_k(x,0)|^{2^*_\alpha} dx=0 \ \text{ or } \ \limsup\limits_{k \to \infty} \int_{B_1} |v_k(x,0)|^{2^*_\alpha} dx \ge \epsilon_0.$$
Since $\epsilon \in (0,\frac{\epsilon_0}{2}),$ this is in contradiction with (\ref{epsilon}),  
thus, $v \not\equiv 0.$ 

To show that  $v \in X^\alpha(\R^{n+1}_+) $ is a  weak solution of (\ref{Main problem.prime}), consider any $\varphi \in C^\infty_0(\R^{n+1}_+),$ and write
\begin{equation}\label{Psi'(v_k,varphi)=o(1) }
\begin{aligned}
o(1) &= \langle \Psi'(v_k) , \varphi \rangle\\ &= k_{\alpha} \int_{\R_+^{n+1}} y^{1-\alpha}  \langle \nabla v_k , \nabla \varphi \rangle dxdy - \gamma \int_{\R^n} \frac{ v_k(x,0) \varphi   }{|x|^{\alpha}}dx \\
&\quad - \int_ {\R^n}  |v_k(x,0)|^{2_{\alpha}^*-2} v_k(x, 0) \varphi   dx - \int_ {\R^n} \frac{ |v_k(x,0)|^{2_{\alpha}^*(s)-2} v_k(x, 0) \varphi  }{|x|^s}dx.
\end{aligned}
\end{equation}
Since $ v_k \rightharpoonup v \text{ in } X^\alpha(\R^{n+1}_+)$ as ${k \to \infty},$ we have that 
$$\int_{\R_+^{n+1}} y^{1-\alpha}  \langle \nabla v_k , \nabla \varphi \rangle dxdy \rightarrow \int_{\R_+^{n+1}} y^{1-\alpha}  \langle \nabla v , \nabla \varphi \rangle dxdy, \quad   \forall \varphi \in C^\infty_0(\R^{n+1}_+).$$

In addition, the boundedness of $v_k$ in $X^\alpha(\R^{n+1}_+) $ yields that $ v_k(.,0),$  $|v_k(.,0)|^{2_{\alpha}^*-2} v_k(.,0) $ and  $ |v_k(.,0)|^{2_{\alpha}^*(s)-2} v_k(.,0) $ are bounded in $L^2(\R^n, |x|^{-\alpha}),$ $ L^{\frac{2_{\alpha}^*}{2_{\alpha}^*-1}}(\R^n)$ and $ L^{\frac{2_{\alpha}^*(s)}{2_{\alpha}^*(s)-1}}(\R^n, |x|^{-s})$ respectively. Therefore, we have the following weak convergence:
\begin{align*}
&v_k(.,0)\rightharpoonup v(.,0) \quad \text{ in } L^2(\R^n, |x|^{-\alpha})\\
&|v_k(.,0)|^{2_{\alpha}^*-2} v_k(.,0) \rightharpoonup |v(.,0)|^{2_{\alpha}^*-2} v(.,0) \quad \text{ in } L^{\frac{2_{\alpha}^*}{2_{\alpha}^*-1}}(\R^n) \\
&|v_k(.,0)|^{2_{\alpha}^*(s)-2} v_k(.,0) \rightharpoonup |v(.,0)|^{2_{\alpha}^*(s)-2} v(.,0) \quad \text{ in }  L^{\frac{2_{\alpha}^*(s)}{2_{\alpha}^*(s)-1}}(\R^n, |x|^{-s}).
\end{align*}
Thus, taking limits as ${k \to \infty}$ in (\ref{Psi'(v_k,varphi)=o(1) }), we obtain that
 
\begin{equation*} \label{Psi'(v, varphi) = 0}
\begin{aligned}
0 &= \langle \Psi'(v) , \varphi \rangle\\ &= k_{\alpha} \int_{\R_+^{n+1}} y^{1-\alpha} \langle \nabla v , \nabla \varphi \rangle dxdy - \gamma \int_{\R^n} \frac{ v(x,0) \varphi   }{|x|^{\alpha}}dx\\
& \quad - \int_ {\R^n} |v(x, 0)|^{2_{\alpha}^*-2} v(x,0)\varphi   dx - \int_ {\R^n} \frac{ |v(x,0)|^{2_{\alpha}^*(s)-2} v(x, 0) \varphi  }{|x|^s}dx . 
\end{aligned}
\end{equation*}
Hence $v$ is a weak solution of (\ref{Main problem.prime}). \\

{\bf Acknowledgments:} Part of this work was done while the authors were visiting the Fields Institute for Research in  Mathematical Sciences (Toronto), during the Thematic program on variational problems in Physics, Economics and Geometry. The authors would like to thank the Fields Institute for its support and its hospitality. 

\end{document}